\documentclass[a4paper,11pt]{article}
\usepackage{mathrsfs, amsmath, amsxtra, amssymb, latexsym, amscd, amsthm}
\usepackage{color, graphicx}
\newtheorem{thm}{Theorem}[section]
\newtheorem{cor}{Corollary}[section]
\newtheorem{lem}{Lemma}[section]
\newtheorem{prop}{Proposition}[section]
\theoremstyle{definition}
\newtheorem{defn}{Definition}[section]
\newtheorem{exam}{Example}[section]

\theoremstyle{remark}
\newtheorem{rem}{Remark}[section]
\numberwithin{equation}{section}
\def\ind{{\rm 1\hspace{-0.90ex}1}}
\def\E{\mathbb{E}}
\begin{document}

\title{Rates of Fisher information convergence in the central limit theorem for nonlinear statistics}
\author{Nguyen Tien Dung\thanks{Department of Mathematics, University of Science, Vietnam National University, Hanoi, 334 Nguyen
Trai, Thanh Xuan, Hanoi, 084 Vietnam. Email: dung@hus.edu.vn}}

\date{\today}          % Ngay
%\address{Department of Mathematics, FPT University\\ Hoa Lac High Tech Park, Hanoi, Vietnam}
%\email{dung\_nguyentien10@yahoo.com, dungnt@fpt.edu.vn}
\maketitle
\begin{abstract}
We develop a general method to study the Fisher information distance in central limit theorem for nonlinear statistics. We first construct completely new representations for the score function. We then use these representations to derive quantitative estimates for the Fisher information distance. To illustrate the applicability of our approach, explicit rates of Fisher information convergence for quadratic forms and the functions of sample means are provided. For the sums of independent random variables, we obtain the Fisher information bounds without requiring the finiteness of Poincar\'e constant. Our method can also be used to bound the Fisher information distance in non-central limit theorems.
\end{abstract}
%\section{Introduction}       % Muc dau tien
%{\large Densities, Tail probabilities and }

\section{Introduction}
% \subsection{Relative Fisher information and previous results}
%\subsection{Motivation}
It is well known that the Fisher information plays a fundamental role in information theory and related fields. Given a random variable $F$ with an absolutely continuous density $p_F$,
the Fisher information of $F$ (or its distribution) is defined by
$$
I(F)= \int_{-\infty}^{+\infty} \frac{p_F'(x)^2} {p_F(x)}dx=\E[\rho_F^2(F)],
$$
where $p_F'$ denotes a Radon-Nikodym derivative of $p_F$ and $\rho_F:=p_F'/p_F$ is the score function. Furthermore, if $\E[F]=\mu$ and ${\rm Var}(F)=\sigma^2,$ the Fisher information distance of $F$ to the normal distribution $Z\sim N(\mu,\sigma^2)$ is defined by
$$I(F\|Z):=\E\left[\left(\rho_F(F)+\frac{F-\mu}{\sigma^2}\right)^2\right]=I(F)-I(Z).$$
We recall that this distance provides a very strong measure of convergence to Gaussian. For example, it dominates the relative entropy (or Kullback-Leibler distance) and the supremum distance between densities. It also dominates the traditional distances such as Kolmogorov distance, total variation distance and Wasserstein distance, etc.

When the density of $F$ is not explicitly known, the estimation of Fisher information  is really difficult. Let us consider the simplest case, that is the case of the normalized sums:
\begin{equation}\label{iklf}
Z_n:= \frac{X_1 + \dots + X_n}{\sqrt{n}},\,\,n\geq 1,
\end{equation}%., see e.g. \cite{B-J} for a short survey.
where $(X_n)_{n \geq 1}$ is a sequence of  independent identically distributed (i.i.d.) random variables with $\E[X_1] = 0$ and ${\rm Var}(X_1) = 1.$ The study of the central limit theorem in information-theoretic sense has a long history beginning in 1959 with the results of Linnik \cite{Linnik1959}. However, even in this simplest case, the problem of obtaining quantitative  bounds on $I(Z_n||Z)$ had been outstanding for many years.

\noindent$\bullet$ The first quantitative results were established in 2004. The authors of two papers \cite{Artstein2004,B-J} require the random variables $X_k's$ to satisfy the Poincar\'e inequality with constant $c > 0$: i.e. for every smooth function $s,$ one has
$$c{\rm Var}(s(X_k))\leq \E[s'(X_k)^2].$$ %(note that the Poincar\'e inequality implies all the moments of $X_1$ are finite).
Under such an assumption, Barron and Johnson \cite{B-J} obtained the following bound
$$I(Z_n\|Z)\leq \frac{2}{2+c(n-1)}I(X_1\|Z).$$
Furthermore, this bound can be extended  to the case of non-identical variables, see Section 3.1 in \cite{Johnson2004} or Remark \ref{jsm4} below. For the weighted sums $\bar{Z}_n:=a_1X_1+\cdots+a_nX_n,$ where $a_1^2+\cdots+a_n^2=1,$ Artstein et al. \cite{Artstein2004} obtained an analogous result that reads
$$I(\bar{Z}_n\|Z)\leq \frac{\alpha(a)}{c/2+(1-c/2)\alpha(a)}I(X_1\|Z),$$
where $\alpha(a):=a_1^4+\cdots+a_n^4.$

\noindent$\bullet$  Ten years later, in their very long paper, Bobkov et al. \cite{Bobkov2014b} proved the following Edgeworth-type expansion: Let $\E |X_1|^s < +\infty$ for an integer $s \geq 2$,
and assume $I(Z_{n_0}) < +\infty$, for some $n_0$. Then for certain
coefficients $c_j$ we have, as $n \rightarrow \infty$,
$$
I(Z_n\|Z) = \frac{c_1}{n} + \frac{c_2}{n^2} + \dots +
\frac{c_{[(s-2)/2)]}}{n^{[(s-2)/2)]}} +
o\left(n^{-\frac{s-2}{2}} \, (\log n)^{-\frac{(s-3)_+}{2}}\right).
$$
\noindent$\bullet$  Recently, Johnson \cite{Johnson2020} used the maximal correlation coefficient to deduce the following result
$$I(Z_n\|Z)\leq \frac{I(X_1\|Z)}{1+\Theta^{(2)}(n-1)},$$
where $\Theta^{(2)}$ is defined by the formula (12) in \cite{Johnson2020}.

Interestingly, the methods developed in the above mentioned papers are very different from each other, see Section I in \cite{Johnson2020} for a discussion. Here we note that the finiteness of Poincar\'e constant $c$ required in \cite{Artstein2004,B-J,Johnson2004}  implies that all moments of $X_k$ are finite. Hence, the results of \cite{Bobkov2014b,Johnson2020} for the case of i.i.d. random variables are remarkable contributions to the literature. Naturally, one may wonder whether the finiteness of Poincar\'e constant can be replaced by weaker assumptions in the case of non-i.i.d. random variables. This question, to the best of our knowledge, is still an open problem. Furthermore, it is well known that most of the random variables occurring in statistics are not sums
of independent random variables, they are often nonlinear statistics of the form
\begin{equation}\label{klf7}
F=F(X_1,\cdots,X_n).
\end{equation}
Over the past decades, several general techniques have been developed to obtain Berry-Esseen error estimates (i.e. bounds on Kolmogorov distance) for $F$ and its special cases, see e.g. \cite{Nicolas2020,Shao2019} and references therein. However, surprisingly, we only find Fisher information bounds in \cite{Artstein2004,Bobkov2014b,B-J,Johnson2004,Johnson2020} for the sums of independent random variables. Besides, by means of the Malliavin-Stein method, some Fisher information bounds were obtained in \cite{Herry2023,Nourdinart2016,Nourdin2014} for Wiener chaos (i.e. when $F$ is a polynomial and $X_1,\cdots,X_n$ are Gaussian random variables).
%Nourdin2014b,

In this paper, motivated by the above observations, our purpose is to develop a new technique for investigating the Fisher information distance in general framework of nonlinear statistics (\ref{klf7}). The generality of our results lies in two facts: (i) we do not require $X_k's$ to be identically distributed, (ii) we do not require any special structure of $F.$ Clearly, this framework is much harder to work with and it requires some new ideas. Note that all the nice properties of the sum of independent random variables are no longer available.

\noindent{\bf About our approach.} By normalizing, we may assume $\E[F]=0$ and ${\rm Var}(F)=1$ for simplicity. The Fisher information distance now becomes
\begin{equation}\label{vail}
I(F\|Z)=\E[(\rho_F(F)+F)^2],
\end{equation}
where $Z\sim N(0,1).$ In the case of the normalized sums $Z_n,$ one can employ the score functions $\rho_k$ of $X_k,1\leq k\leq n$ to represent the score function of $Z_n.$ In fact, we have
\begin{equation}\label{g6vail}\rho_{Z_n}(Z_n)=\E\left[\frac{\rho_{1}(X_1)+\cdots+\rho_{n}(X_n)}{\sqrt{n}}\big|Z_n\right],\,\,n\geq 1.
\end{equation}
The representation (\ref{g6vail}) and its variants have been widely used in the previous works by several authors. However, unfortunately, we cannot combine this representation with (\ref{vail}) to derive an effective bound on $I(Z_n\|Z)$ when $X_1$ is non-Gaussian. Indeed, we have
$$I(Z_n\|Z)\leq \E\left|\frac{\rho_{1}(X_1)+\cdots+\rho_{n}(X_n)}{\sqrt{n}}+Z_n\right|^2=\E|\rho_{1}(X_1)+X_1|^2\nrightarrow 0\,\,\text{as}\,\,n\to\infty.$$
This observation suggests that, to bound the Fisher information distance directly from its definition (\ref{vail}),  we have to seek a new representation for the score function. This is exactly what we want to do in the present paper. The key idea in our approach is to use the following Stein-type covariance formula.
\begin{lem}\label{kflo} Let $X$ be a random variable such that its distribution has a density $p$ supported on a connected interval of $\mathbb{R}.$ Also let $\alpha$ be a bounded and differentiable function. Then, for any function $\beta$ with $\E|\beta(X)|<\infty,$ we have
\begin{equation}\label{key}
{\rm Cov}(\alpha(X),\beta(X))=\int_{-\infty}^\infty \alpha'(x)\left(\int_x^\infty (\beta(y)-\E[\beta(X)])p(y)dy\right)dx.
\end{equation}
\end{lem}
\noindent The above lemma is not new and its proof is really simple, see e.g. \cite{Ernst2020,Saumard2018,Saumard2019}. Indeed, we have
\begin{align*}
{\rm Cov}(\alpha(X),\beta(X))&=\int_{-\infty}^\infty\alpha(x)(\beta(x)-\E[\beta(X)])p(x)dx\notag\\
&=-\int_{-\infty}^\infty \alpha(x)d\left(\int_x^\infty (\beta(y)-\E[\beta(X)])p(y)dy\right)\notag\\
&=\int_{-\infty}^\infty \alpha'(x)\left(\int_x^\infty (\beta(y)-\E[\beta(X)])p(y)dy\right)dx.
\end{align*}
Notice that $\lim\limits_{x\to\pm\infty}\int_x^\infty (\beta(y)-\E[\beta(X)])p(y)dy=0$ because $\E|\beta(X)|<\infty.$ At first glance we do not see any connection between Lemma \ref{kflo} and the score function $\rho_F$ of $F.$ However, Lemma \ref{kflo} is very useful to find a random variable $H$ such that
$$\E[f'(F)]=\E[f(F)H]$$
for any $f$ belonging to a  suitable class of test functions. The above relation, together with standard arguments, allows us to recover the density function $p_F$ of $F$ as follows
$$p_F(x)=\E[\ind_{\{F>x\}}H],$$
and hence, the score function $\rho_F(F)=-\E[H|F].$ Consequently, we obtain the Fisher information bound
$$I(F\|Z)\leq \E|H-F|^2.$$
The rest of this article is organized as follows. In Section \ref{klf5}, we introduce some notations and definitions. Sections \ref{89jk} and \ref{89jk2} are devoted to our main results. In Section \ref{89jk}, we provide two abstract bounds on the Fisher information distance: the first one uses martingale-type decompositions, see Theorem \ref{gko} and the second one uses approximate statistics, see Theorem \ref{8uh4}.

In Section \ref{89jk2}, we apply the bounds of Section \ref{89jk} to obtain explicit Fisher information bounds in the central limit theorem for Quadratic forms and the functions of sample means, see Theorems \ref{uij8} and \ref{kgl1}, respectively. Under suitable assumptions, these bounds are consistent with known Berry-Esseen estimates for the rate of weak convergence, see Corollaries \ref{fkm9}, \ref{fdv} and \ref{kmdl}. Particularly, for the normalized sum of non-i.i.d random variables, we obtain the rate $O(1/n)$ without requiring the finiteness of Poincar\'e constants, see Proposition \ref{9ohj2} and Theorem \ref{io9y}.

The conclusion and further remarks are given in Section \ref{jkfm}. We show, in Subsection \ref{ufk2}, that ones can combine the Fisher information bounds with concentration inequalities to derive nonuniform Berry-Esseen bounds. When applying to the normalized sums, for the first time, a Gaussian nonuniform Berry-Esseen bound is obtained. In Subsection \ref{ufk2b}, we show that our method can also be used to bound the Fisher information distance in non-central limit theorems.

%The explicit estimates for the rate of convergence are provided in As an illustrative example,

\section{Notations and Stein kernels}\label{klf5}
%\section{Assumption and notations}
We let $(\Omega,\mathcal{F},\mathbb{P})$ denote a suitable probability space on which all random quantities subsequently dealt with are defined. Throughout this paper, we use the following notations.

\noindent $\bullet$ We denote by  $(X_k)_{k\geq 1}$ the sequence of independent real-valued random variables (not necessarily identically distributed). Without loss of generality, we assume that $\E[X_k]=0$ and ${\rm Var}(X_k)=1$ for all $k\geq 1.$ Since we are interested in the Fisher information distance, we further assume that, for each $k\geq 1,$ the  distribution  of  random  variable $X_k$ has a differentiable density supported on some connected interval of $\mathbb{R}.$ %and has an a.e. positive  density  on that interval.$[a_k,b_k]$

\noindent $\bullet$ For each $k\geq 1,$ we denote by $p_k$ the density function of $X_k$ and by $\E_k$ the expectation with respect to $X_k.$

%we denote by $\mathcal{C}^k$ the space of $k$-times continuously differentiable functions $f:\mathbb{R}^n\to \mathbb{R}.$
\noindent $\bullet$ For any integers $k\geq 0,n\geq 1,$  we let ${\mathbb{X}_n}$ be the random vector $(X_1,\cdots,X_n)$ and denote by $\mathcal{C}_{\mathbb{X}_n}^k$ the space of random variables defined by
  $$\mathcal{C}_{\mathbb{X}_n}^k=\{F({\mathbb{X}_n}):F\,\, \text{is $k$-times continuously differentiable function on the range of $\mathbb{X}_n$}\},$$
and, with a slight abuse of notation, we write $F$ instead of $F({\mathbb{X}_n}).$

\noindent $\bullet$ As usual, $L^p(\Omega)$ denotes the space of random variables $X$ such that $\|X\|_p:=(\E|X|^p)^{\frac{1}{p}}<\infty.$

\noindent $\bullet$ For any function $h$ on $\mathbb{R}^n$ and $x=(x_1,\cdots,x_n)\in \mathbb{R}^n,$ we write $$h(x,x_k=y)=h(x_1,\cdots,x_{k-1},y,x_{k+1},\cdots,x_n).$$ If the derivatives exist, we set $\partial_kh=\frac{\partial h}{\partial x_k},$ $\partial_{ki}h=\frac{\partial^2 h}{\partial x_i\partial x_k}$ and so on.

We now define the random quantities that will be used to formulate our main results.%and cumulative distribution function $\nu,$
\begin{defn}\label{klfa} Given a random variable $Y\in L^1(\Omega)$ with a density $p_Y$ supported on some connected interval of $\mathbb{R},$  we define the function $\tau_Y$ associated to  $Y$ as follows
$$\tau_Y(x)=\frac{\int_{x}^\infty (y-\E[Y])p_Y(y)dy}{p_Y(x)}$$
for $x$ in the support of $p_Y.$ When $Y=X_k,$ we write $\tau_k$ instead of $\tau_{X_k}.$
\end{defn}
%By its definition, we have $\tau_Y(Y)>0\,\,a.s.$ For random variables $Y$ with $\E|Y|^2<\infty,$ it follows from  Lemma \ref{kflo} that $\E[\tau_Y(Y)]={\rm Var}(Y).$
In the literature, the function $\tau_Y$ is the so-called Stein kernel of $Y.$ Note that the Stein kernel was first studied in \cite{Stein1986}, and since then, it has became an effective tool in comparing of distributions, see e.g. \cite{Bonis2020,Ernst2020,Fathi2019,Ley2017,Mijoule2023,Saumard2019}. In the present paper, we use the following two fundamental properties of Stein kernels:%\cite{Bonis2020,Fathi2019,Ernst2020,Mijoule2023,Saumard2019}
$$\tau_Y(Y)\geq 0\,\,\,\text{and}\,\,\,\E[\tau_Y(Y)]={\rm Var}(Y).$$
In particular, we always have $\E[\tau_k(X_k)]=1$ for all $k\geq 1.$

\noindent Generalizing the definition of Stein kernels, we have
\begin{defn} Let $F\in \mathcal{C}_{\mathbb{X}_n}^0\cap L^1(\Omega).$  For each $k=1,\cdots,n,$ we define
%$$\mathcal{L}_kF=\frac{\int_{-\infty}^\infty (\nu_k(X_k\wedge y)-\nu_k(X_k)\nu_k(y))\partial_kF({\mathbb{X}_n},X_k=y)dy}{p_k(X_k)}.$$
$$\mathcal{L}_kF=\mathcal{L}_kF({\mathbb{X}_n}):=\frac{\int_{X_k}^\infty \big(F({\mathbb{X}_n},X_k=x)-\E_k[F]\big)p_k(x)dx}{p_k(X_k)}.$$
\end{defn}
In Section \ref{89jk2} below, we require the Stein kernels to be bounded from above and/or from below. Let us provide some examples of Stein kernels so that the reader can check such conditions.
\begin{exam}[Pearson distributions]  By definition, the law of a random variable $Y$ is a member of the Pearson family
of distributions if the density $p_Y$ of $Y$ is supported on $(a,b),$ where $-\infty\leq a<b\leq \infty$ and satisfies the differential equation
$$\frac{p'_Y(x)}{p_Y(x)}=-\frac{mx+k}{\alpha_1 x^2+\alpha_2 x+\alpha_3},\,\,x\in (a,b),$$
for some real numbers $m,k,\alpha_1,\alpha_2,\alpha_3.$ If furthermore $\E[Y]=0$ and $m=2\alpha_1+1,k=\alpha_2$ then the Stein kernel of $Y$ can be computed by (see page 65 in \cite{Stein1986})
$$\tau_Y(Y)=\alpha_1 Y^2+\alpha_2 Y+\alpha_3.$$
One can verify that this is the case of the following distributions, among others: Normal distribution, Student's $t$-distribution, Gamma distribution,  Chi-squared distribution,  Exponential distribution, $F$-distribution, Beta distribution, Continuous uniform distribution, Generalized Cauchy distribution.

Clearly, $\tau_Y(Y)$ is bounded from above if the support of $p_Y$ is a finite interval. On the other hand, $\tau_Y(Y)$ is bounded from below if $\alpha_1\geq 0,\alpha_2=0$ and $\alpha_3>0.$
%Inverse-gamma distribution,Inverse-chi-squared distribution,
%\begin{enumerate}
%\item Normal distribution with variance $\sigma^2:$ $\tau_Y(Y)=\sigma^2$
%\item Generalized Cauchy distribution $p(x)=c(1+x^2)^{-\beta}$ $\tau_Y(Y)=\frac{1+Y^2}{2\beta-2}$
%\item Uniform distribution on $[0,1]:$ $\tau_Y(Y)=\frac{1}{2}Y(1-Y)$
%\item Beta distribution with parameters $\alpha,\beta>0:$ $\tau_Y(Y)=(\alpha+\beta)^{-1}Y(1-Y)$
%\end{enumerate}
\end{exam}
\begin{exam}[functions of a random variable] Let $Y$ be as in Definition \ref{klfa}.  We consider the random variable $X:=f(Y),$ where $f:\mathbb{R}\to \mathbb{R}$ is a differentiable function such that $f'>0,$ $\E[f(Y)]=0$ and $\E|f(Y)|<\infty.$
By straightforward computations, we have
$$\tau_{X}(X)=f'(f^{-1}(X))\frac{\int_{f^{-1}(X)}^\infty f(y)p_Y(y)dy}{p_Y(f^{-1}(X))}.$$
Then, by using Corollary 2.1 in \cite{Saumard2018}, we have the following relation
\begin{equation}\label{pnj}
(\inf\limits_{x\in \mathbb{R}}f'(x))^2\tau_Y(f^{-1}(X))\leq \tau_{X}(X)\leq (\sup\limits_{x\in \mathbb{R}}f'(x))^2\tau_Y(f^{-1}(X)).\end{equation}
\end{exam}
\begin{exam}[Gaussian functionals] Let $Z=(Z_1,\cdots,Z_n)$ be a standard Gaussian vector. Also let $h:\mathbb R^n \to \mathbb R$ be a differentiable function such that $h$ and their derivatives have subexponential growth at infinity. We consider the random variable $Y=h(Z)$ and define
$$u(z):=\int_0^1\E\left[\sum\limits_{k=1}^n\partial_k h(z)\partial_k h(\alpha z+\sqrt{1-\alpha^2}Z)\right]d\alpha,\,\,\,z\in \mathbb{R}^n.$$
It follows from Lemma 2.1.4 in \cite{Adler2007} that
$$\E[f(Y)(Y-\E[Y])]= E[f'(Y)u(Z)]$$
for any differentiable function $f:\mathbb{R}\to \mathbb{R}$ with compact support. On the other hand, by Lemma \ref{kflo}, we have
$$\E[f(Y)(Y-\E[Y])]= E[f'(Y)\tau_Y(Y)].$$
So we get the following representation for Stein kernel
$$\tau_Y(Y)=\E[u(Z)|Y].$$
Consequently, $\tau_Y(Y)$ is bounded from above if $\sup\limits_{z\in \mathbb{R}^n}|\partial_k h(z)|<\infty$ and $\tau_Y(Y)$ is bounded from below if $\inf\limits_{z\in \mathbb{R}^n}\partial_k h(z)>0$ for all $1\leq k\leq n.$
\end{exam}
%\begin{exam}[Malliavin differentiable random variables]
%
%\end{exam}

\begin{rem}Let $\phi$ be the density of standard normal random variable $Z$ and $\Phi(x)=\int_{-\infty}^x\phi(y)dy.$ We can write $X_k=f(Z):=\nu_k^{-1}(\Phi(Z)),$ where $\nu_k$ denotes the cumulative distribution function of $X_k.$ We assume that the ratio of the densities satisfies
$$c\leq \frac{\phi(\Phi^{-1}(y))}{p_k(\nu_k^{-1}(y))}\leq C\,\,\forall\,y\in [0,1]$$
for some $C\geq c>0.$ Note that $\tau_Z(x)=1$ and $f'(x)=\frac{\phi(x)}{p_k(\nu_k^{-1}(\Phi(x)))},\,\,x\in \mathbb{R}.$ Hence, it follows from the relation (\ref{pnj}) that
$$c^2\leq \tau_k(X_k)\leq C^2.$$
%$\beta:=\max\big\{\sup\limits_{k\geq 1}\E|X_k|^8,\sup\limits_{k\geq 1}\E|\tau_k(X_k)|^8,\sup\limits_{k\geq 1}\E|\tau'_k(X_k)|^8\big\}<\infty,$
\end{rem}

\section{General Fisher information bounds}\label{89jk}
In this section, we provide two general bounds on the Fisher information distance $I(F\|Z)$ of nonlinear statistics $F\in \mathcal{C}_{\mathbb{X}_n}^2.$ Hereafter, we always denote by $Z$ the standard normal random variable.
\subsection{Bounds using martingale-type decomposition}
Our idea here is to write $F$ as a sum of random variables, and then apply the covariance formula (\ref{key}) with respect to each random variable $X_k.$
Let us start with the following definition.%where $m,p\geq 1$ are integer numbers.
\begin{defn}\label{hkd}Let $F\in \mathcal{C}_{\mathbb{X}_n}^0\cap L^1(\Omega).$  The set $\mathcal{M}:=\{\mathcal{M}_1F,\cdots,\mathcal{M}_nF\}$ is called a decomposition of $F$ if

$(i)$ $F-\E[F]=\mathcal{M}_1F+\cdots+\mathcal{M}_nF,$

$(ii)$ $\E_k[\mathcal{M}_kF]=0$ for all $1\leq k\leq n,$ %where $\E_k$ denotes the expectation with respect to $X_k.$

$(iii)$  $\mathcal{M}_kF\in \mathcal{C}_{\mathbb{X}_n}^0\cap L^1(\Omega)$ for all $1\leq k\leq n.$

\noindent We say $\mathcal{M}$ is a differentiable decomposition if it is a decomposition of $F$ and, in addition, we have $\mathcal{M}_kF\in \mathcal{C}_{\mathbb{X}_n}^1$ and $\partial_i\mathcal{M}_kF\in L^1(\Omega)$ for all $1\leq i,k\leq n.$
\end{defn}
Basically, the decomposition of $F$ always exists and is not unique. Indeed, we have the following backward and forward martingale decompositions
$$F-\E[F]=\sum\limits_{k=1}^n\E\left[F-\E_k[F]|X_1,...,X_k\right],$$
$$F-\E[F]=\sum\limits_{k=1}^n\E\left[F-\E_k[F]|X_k,...,X_n\right].$$
Thus, for any $\alpha\in \mathbb{R},$ the functions
$$\mathcal{M}_kF=\mathcal{M}_kF({\mathbb{X}_n}):=\alpha\E\left[F-\E_k[F]|X_1,...,X_k\right]+(1-\alpha)\E\left[F-\E_k[F]|X_k,...,X_n\right],\,\,1\leq k\leq n,$$
form a decomposition of $F.$ Furthermore, this decomposition is differentiable if $F\in \mathcal{C}_{\mathbb{X}_n}^1$ and $\partial_iF\in L^1(\Omega)$ for all $1\leq i\leq n.$
\begin{defn}Let $F\in \mathcal{C}_{\mathbb{X}_n}^1\cap L^1(\Omega)$ and $G\in \mathcal{C}_{\mathbb{X}_n}^0\cap L^1(\Omega).$ Given a decomposition $\{\mathcal{M}_1G,\cdots,\mathcal{M}_nG\}$ of $G,$ we define
$$\Theta_{F,G}=\Theta_{F,G}({\mathbb{X}_n}):=\sum\limits_{k=1}^n\partial_k F\mathcal{L}_k\mathcal{M}_kG.$$
When $F=G,$ we write  $\Theta_{F}$ instead of $\Theta_{F,F}.$
\end{defn}
We have the following representation for the score function of nonlinear statistics.
\begin{prop}\label{klf2} Let $F\in \mathcal{C}_{\mathbb{X}_n}^2\cap L^1(\Omega)$ with $\E[F]=0.$ Suppose that there exists a differentiable decomposition $\mathcal{M}=\{\mathcal{M}_1F,\cdots,\mathcal{M}_nF\}$ of $F$ such that
$\Theta_F\neq 0$ a.s. and $\frac{F}{\Theta_F},\frac{\Theta_{\Theta_F,F}}{\Theta_F^2}\in L^1(\Omega).$ Then the law of $F$ has an absolutely continuous density given by
\begin{equation}\label{d3}
p_F(x)=\E\left[\ind_{\{F>x\}}\left(\frac{F}{\Theta_F}+\frac{\Theta_{\Theta_F,F}}{\Theta_F^2}\right)\right].
\end{equation}
Moreover, we have the following representation for the score function
\begin{equation}\label{d3t13}
\rho_F(x):=p'_F(x)/p_F(x)=-\E\left[\frac{F}{\Theta_F}+\frac{\Theta_{\Theta_F,F}}{\Theta_F^2}\big|F=x\right].
\end{equation}
\end{prop}
\begin{proof} The proof is broken up into two steps.

\noindent{\it Step 1.} We claim that
\begin{align}
\E[f'(F)]&=\E\left[f(F)\left(\frac{F}{\Theta_F}+\frac{\Theta_{\Theta_F,F}}{\Theta_F^2}\right)\right]\label{lol2}
\end{align}
for any bounded and differentiable function $f:\mathbb{R}\to \mathbb{R}$ with bounded derivative.

Since $\mathcal{M}$ is a differentiable decomposition, this implies that the function $\Theta_F$ is differentiable. Hence, for every $\varepsilon>0,$ the functions $\alpha(x):=\frac{f(F)\Theta_F}{\Theta_F^2+\varepsilon}({\mathbb{X}_n},X_k=x)$ and $\beta(x):=\mathcal{M}_kF({\mathbb{X}_n},X_k=x)$ satisfy the conditions of Lemma \ref{kflo}. Applying the covariance formula (\ref{key}) yields
\begin{align*}
&\E_k\left[\frac{f(F)\Theta_F}{\Theta_F^2+\varepsilon}\mathcal{M}_kF\right]=\int_{-\infty}^{\infty}\left(\frac{f'(F)\Theta_F\partial_k F+f(F)\partial_k\Theta_F}{\Theta_F^2+\varepsilon}-\frac{2f(F)\Theta_F^2\partial_k\Theta_F}{(\Theta_F^2+\varepsilon)^2}\right)({\mathbb{X}_n},X_k=x)\\
&\hspace{9cm}\times
\left(\int_{x}^{\infty} \mathcal{M}_kF({\mathbb{X}_n},X_k=y)dy\right)dx\\
&=\E_k\left[\left(\frac{f'(F)\Theta_F\partial_k F+f(F)\partial_k\Theta_F}{\Theta_F^2+\varepsilon}-\frac{2f(F)\Theta_F^2\partial_k\Theta_F}{(\Theta_F^2+\varepsilon)^2}\right)\mathcal{L}_k\mathcal{M}_kF\right]\\
&=\E_k\left[\frac{f'(F)\Theta_F}{\Theta_F^2+\varepsilon}\partial_k F\mathcal{L}_k\mathcal{M}_kF\right]-\E_k\left[\frac{f(F)(\Theta_F^2-\varepsilon)}{(\Theta_F^2+\varepsilon)^2} \partial_k\Theta_F \mathcal{L}_k\mathcal{M}_kF\right],\,\,1\leq k\leq n.
\end{align*}
Hence, by the independence of random variables $X_k's,$ we deduce
\begin{align*}
\E\left[\frac{f(F)\Theta_FF}{\Theta_F^2+\varepsilon}\right]&=\E\left[\frac{f(F)\Theta_F\sum\limits_{k=1}^n \mathcal{M}_kF}{\Theta_F^2+\varepsilon}\right]
=\sum\limits_{k=1}^n\E\left[\E_k\left[\frac{f(F)\Theta_F}{\Theta_F^2+\varepsilon}\mathcal{M}_kF\right]\right]\\
&=\sum\limits_{k=1}^n\E\left[\frac{f'(F)\Theta_F}{\Theta_F^2+\varepsilon}\partial_k F\mathcal{L}_k\mathcal{M}_kF\right]-\sum\limits_{k=1}^n\E\left[\frac{f(F)(\Theta_F^2-\varepsilon)}{(\Theta_F^2+\varepsilon)^2}\partial_k\Theta_F \mathcal{L}_k\mathcal{M}_kF\right]\\
&=\E\left[f'(F)\frac{\Theta_F^2}{\Theta_F^2+\varepsilon}\right]-\E\left[f(F)\frac{(\Theta_F^2-\varepsilon)\Theta_{\Theta_F,F}}{(\Theta_F^2+\varepsilon)^2}\right].
\end{align*}
We let $\varepsilon\to 0$ and, by the dominated convergence theorem, we obtain
$$\E\left[\frac{f(F)F}{\Theta_F}\right]=
\E\left[f'(F)\right]-\E\left[f(F)\frac{\Theta_{\Theta_F,F}}{\Theta_F^2}\right].$$
Rearranging the terms gives us the claim (\ref{lol2}).

\noindent{\it Step 2.} We take $f(x)=\int_{-\infty}^xg(y)dy,$ where  $g$ is a continuous function with compact support.  Note that $f$ is necessarily bounded. Using the relation (\ref{lol2}) and Fubini's theorem we obtain
\begin{align*}
\E[g(F)]&=\E\left[\int_{-\infty}^Fg(x)dx\left(\frac{F}{\Theta_F}+\frac{\Theta_{\Theta_F,F}}{\Theta_F^2}\right)\right]\\
&=\int_{-\infty}^\infty g(x)\E\left[\ind_{\{F>x\}}\left(\frac{F}{\Theta_F}+\frac{\Theta_{\Theta_F,F}}{\Theta_F^2}\right)\right]dx.
\end{align*}
This implies that the density $p_F$ of $F$ exists and is given by the representation (\ref{d3}). Furthermore, we can rewrite (\ref{d3}) as follows
\begin{align*}
p_{F}\left( x\right)&=\E\left[ \mathbf{1}_{\left\{ F>x\right\} }\E\left[\frac{F}{\Theta_F}+\frac{\Theta_{\Theta_F,F}}{\Theta_F^2}\big| F\right]\right]\\
%&=\E\left[ \mathbf{1}_{\left\{ F>x\right\} }w(F)\right]\\
&=\int_x^\infty \E\left[\frac{F}{\Theta_F}+\frac{\Theta_{\Theta_F,F}}{\Theta_F^2}\big| F=y\right]p_{F}(y)dy.
\end{align*}
So $p_F$ is absolutely continuous and the representation (\ref{d3t13}) is verified. The proof of the proposition is complete.
\end{proof}
\begin{rem}The reader can consult \cite{Courtade2016,Nourdin2014,Nourdin2014b,Tulino2006} for various representations of the score function. Note that our formula (\ref{d3t13}) is different from the representations used there.
%have some similarities with the recursive formula developed by Privault in[25] \cite{Nourdinart2016}
\end{rem}
\begin{rem}Using the same arguments as in the proof of (\ref{lol2}) we have
$$\E[f(F)F]=\sum\limits_{k=1}^n \E\left[f(F)\mathcal{M}_kF\right]=\E[f'(F)\Theta_F]$$
for any bounded and differentiable function $f:\mathbb{R}\to \mathbb{R}$ with bounded derivative. Then, we can employ the method developed in \cite{Nourdin2009} to derive the following representation for the density function
$$p_F(x)=\frac{\E|F|}{2g_F(x)}\exp\left(-\int_{0}^z \frac{u}{g_F(u)}du\right),
 $$
 where $g_F(F)=\E[\Theta_F|F].$ However, it is not easy to use this presentation for computing the score function $\rho_F(x).$ This is due to the fact that we do not know how to compute $g'_F(x).$
\end{rem}
As discussed in the introduction, the key step in our approach is to construct a suitable representation for the score function. Let us provide some simple examples to show why the formula (\ref{d3t13}) is a suitable representation for our purpose.
\begin{exam}We assume, for simplicity, the random variables $X_k's$ are identically distributed.

\noindent$(a)$ We consider the normalized sums $Z_n$ defined by (\ref{iklf}). Using the decomposition $\mathcal{M}_kZ_n=\frac{X_k}{\sqrt{n}},1\leq k\leq n$ we obtain
$$\Theta_{Z_n}= \frac{\tau_1(X_1) + \cdots + \tau_n(X_n)}{n}$$
and
$$\Theta_{\Theta_{Z_n},Z_n}= \frac{\tau'_1(X_1)\tau_1(X_1) + \cdots + \tau'_n(X_1)\tau_n(X_n)}{n\sqrt{n}}$$
By the law of large numbers, we have $\Theta_{Z_n}\to 1$ and $\Theta_{\Theta_{Z_n},Z_n}\to 0$ as $n\to\infty.$ Hence, intuitively, $\rho_{Z_n}(Z_n)$ becomes closer to $-Z_n$ as $n$ increases. So we can predict that $Z_n\to Z$ (note that a random variable $X$ is standard normal if and only if its score function is linear, $\rho_X(X)=-X$) and we expect to have $I(Z_n\|Z)\to 0$ as $n\to\infty.$

\noindent$(b)$ We consider the following sums of dependent random variables
$$F_n:=\frac{Y_1+\cdots+Y_n}{\sqrt{n}},\,\,n\geq 1,$$
where $Y_k:=X_kX_{k+1},k\geq 1.$ We can decompose $F_n$ as follows
$$\mathcal{M}_1F_n=\frac{X_1X_2}{2\sqrt{n}},\,\,\,\mathcal{M}_{n+1}F_n=\frac{X_nX_{n+1}}{2\sqrt{n}},$$
$$\mathcal{M}_kF_n=\frac{X_{k-1}X_k+X_kX_{k+1}}{2\sqrt{n}},\,\,2\leq k\leq n.$$
Using this decomposition we obtain
$$\Theta_{F_n}=\frac{1}{2n}X_2^2\tau_1(X_1)+\frac{1}{2n}\sum\limits_{k=2}^n(X_{k-1}+X_{k+1})^2\tau_k(X_k)
+\frac{1}{2n}X_n^2\tau_{n+1}(X_{n+1}).$$
We observe that $\Theta_{F_n}$ is a sum of $2$-dependent random variables. Hence, by the law of large numbers, $\Theta_{F_n}\to 1$  as $n\to\infty.$ Similarly, we also have $\Theta_{\Theta_{F_n},F_n}\to 0$ as $n\to\infty.$ So, once again, the score function $\rho_{F_n}(F_n)$ becomes closer to $-F_n$ as $n$ increases and we expect to have $I(F_n\|Z)\to 0$ as $n\to\infty.$
\end{exam}
We now are in a position to provide a general bound on the Fisher information distance.
\begin{thm}\label{gko}
Let $F\in \mathcal{C}_{\mathbb{X}_n}^2$ be such that $\E[F]=0$ and ${\rm Var}(F)=1.$ Then, for any differentiable decomposition $\{\mathcal{M}_1F,\cdots,\mathcal{M}_nF\}$ of $F$ with $\Theta_F\neq 0$ a.s. and for $p,q,r,s,t>1$ with $\frac{1}{p}+\frac{1}{q}+\frac{1}{r}=\frac{1}{s}+\frac{1}{t}=1,$ we have
\begin{equation}\label{lfm3}
I(F\|Z)\leq 2\|F\|_{2p}^2\|\Theta_F^{-1}\|_{2q}^2\|1-\Theta_F\|_{2r}^2+2\bigg\|\frac{\sum\limits_{k=1}^n\frac{|\mathcal{L}_k\mathcal{M}_kF|^2}{\tau_k(X_k)}}{\Theta_F^4}\bigg\|_s
\big\|\sum\limits_{k=1}^n|\partial_k\Theta_F|^2\tau_k(X_k)\big\|_t,
%I(F\|Z)\leq 2\bigg\|\frac{F}{\Theta_F}\bigg\|_4^2\|1-\Theta_F\|_4^2+2\bigg\|\frac{\sum\limits_{k=1}^n\frac{|\mathcal{L}_k\mathcal{M}_kF|^2}{\tau_k(X_k)}}{\Theta_F^4}\bigg\|_2
%\big\|\sum\limits_{k=1}^n|\partial_k\Theta_F|^2\tau_k(X_k)\big\|_2,
\end{equation}
provided that the norms in the right hand side are finite.
\end{thm}
\begin{proof}By the representation formula (\ref{d3t13}) we deduce
\begin{align}
I(F\|Z)&=\E[(\rho_F(F)+F)^2]\notag\\
&\leq\E\left|\frac{F}{\Theta_F}+\frac{\Theta_{\Theta_F,F}}{\Theta_F^2}-F\right|^2\notag\\
&\leq 2\E\left|\frac{F(1-\Theta_F)}{\Theta_F}\right|^2+2\E\left|\frac{\Theta_{\Theta_F,F}}{\Theta_F^2}\right|^2.\label{jkfo9}
\end{align}
By H\"older's inequality we have
\begin{equation}\label{jkfo9a}
\E\left|\frac{F(1-\Theta_F)}{\Theta_F}\right|^2\leq \|F\|_{2p}^2\|\Theta_F^{-1}\|_{2q}^2\|1-\Theta_F\|_{2r}^2.
\end{equation}
By using the Cauchy-Schwarz inequality we obtain
\begin{align*}
|\Theta_{\Theta_F,F}|=\big|\sum\limits_{k=1}^n\partial_k\Theta_F\mathcal{L}_k\mathcal{M}_kF\big|
\leq\sqrt{\sum\limits_{k=1}^n|\partial_k\Theta_F|^2\tau_k(X_k)\sum\limits_{k=1}^n\frac{|\mathcal{L}_k\mathcal{M}_kF|^2}{\tau_k(X_k)}},
\end{align*}
and hence,
\begin{align}
\E\left|\frac{\Theta_{\Theta_F,F}}{\Theta_F^2}\right|^2&\leq \E\bigg[\sum\limits_{k=1}^n|\partial_k\Theta_F|^2\tau_k(X_k)\frac{\sum\limits_{k=1}^n\frac{|\mathcal{L}_k\mathcal{M}_kF|^2}{\tau_k(X_k)}}{\Theta_F^4}
\bigg]\notag \\
&\leq \bigg\|\frac{\sum\limits_{k=1}^n\frac{|\mathcal{L}_k\mathcal{M}_kF|^2}{\tau_k(X_k)}}{\Theta_F^4}\bigg\|_s
\big\|\sum\limits_{k=1}^n|\partial_k\Theta_F|^2\tau_k(X_k)\big\|_t.\notag
\end{align}
This, together with (\ref{jkfo9}) and (\ref{jkfo9a}), gives us the bound (\ref{lfm3}). The proof of the theorem is complete.
\end{proof}
\subsection{Bounds using approximate statistics}
In the previous subsection, the Fisher information bound strongly depends on the behavior of $\Theta_F.$ However, when $F$ is a strong nonlinear statistic, it is not easy to compute $\Theta_F.$ In this subsection, we give an attempt to handle such cases. Our idea is to approximate $F$ by a nicer statistic. It will be an interesting problem to extend the method to other approximations. Here we simply consider the case where $F$ can be approximated by a sum of independent random variables, that is,
$$F\approx\sum\limits_{k=1}^n g_k(X_k)\,\,\text{in}\,\,L^2(\Omega).$$
Note that this is the case of many fundamental statistics such as the functions of sample means, $U$-statistics and $L$-statistics, etc. Now, similarly as in the previous subsection, we first construct the score function of $F.$ We have the following.
\begin{prop}\label{klf22} Let $F\in \mathcal{C}_{\mathbb{X}_n}^2\cap L^2(\Omega).$  Given continuous functions $g_k,1\leq k\leq n$ satisfying $\E|g_k(X_k)|<\infty$ and $\E[g_k(X_k)]=0,$ we define
$$\nabla_F:=\sum\limits_{k=1}^n\partial_k F\mathcal{L}_kg_k,$$
where $\mathcal{L}_kg_k=\mathcal{L}_kg_k(X_k):=\frac{\int_{X_k}^\infty g_k(x)p_k(x)dx}{p_k(X_k)}.$ Assume that $\nabla_F\neq 0$ a.s. and the random variables $\frac{\sum\limits_{k=1}^n g_k(X_k)}{\nabla_F},\frac{\sum\limits_{k=1}^n \partial_k\nabla_F\mathcal{L}_kg_k}{\nabla_F^2}$ are in $L^1(\Omega).$ Then the law of $F$ has an absolutely continuous density given by
\begin{equation}\label{d3sa}
p_F(x)=\E\left[\ind_{\{F>x\}}\left(\frac{\sum\limits_{k=1}^n g_k(X_k)}{\nabla_F}+\frac{\sum\limits_{k=1}^n\partial_k\nabla_F\mathcal{L}_kg_k}{\nabla_F^2}\right)\right].
\end{equation}
Moreover, we have
\begin{equation}\label{d3t13b}
\rho_F(x)=-\E\left[\frac{\sum\limits_{k=1}^n g_k(X_k)}{\nabla_F}+\frac{\sum\limits_{k=1}^n \partial_k\nabla_F\mathcal{L}_kg_k}{\nabla_F^2}\bigg|F=x\right].
\end{equation}
\end{prop}
\begin{proof} Let $f:\mathbb{R}\to \mathbb{R}$ be a bounded and differentiable function with bounded derivative. Fixed $\varepsilon>0,$ we apply Lemma \ref{kflo} to the functions $\alpha(x)=\frac{f(F)\nabla_F}{\nabla_F^2+\varepsilon}({\mathbb{X}_n},X_k=x),\beta(x)=g_k(x)$ and we obtain
\begin{align*}
&\E_k\left[\frac{f(F)\nabla_F}{\nabla_F^2+\varepsilon}g_k(X_k)\right]\\
&=\E_k\left[\left(\frac{f'(F)\nabla_F\partial_k F+f(F)\partial_k\nabla_F}{\nabla_F^2+\varepsilon}-\frac{2f(F)\nabla_F^2\partial_k\nabla_F}{(\nabla_F^2+\varepsilon)^2}\right)\mathcal{L}_kg_k\right]\\
&=\E_k\left[\frac{f'(F)\nabla_F}{\nabla_F^2+\varepsilon}\partial_k F\mathcal{L}_kg_k\right]-\E_k\left[\frac{f(F)(\nabla_F^2-\varepsilon)}{(\nabla_F^2+\varepsilon)^2} \partial_k\nabla_F \mathcal{L}_kg_k\right],\,\,1\leq k\leq n.
\end{align*}
As a consequence, we deduce
\begin{align*}
\E\left[\frac{f(F)\nabla_F\sum\limits_{k=1}^n g_k(X_k)}{\nabla_F^2+\varepsilon}\right]&=\sum\limits_{k=1}^n\E\left[\E_k\left[\frac{f(F)\nabla_F}{\nabla_F^2+\varepsilon}g_k(X_k)\right]\right]\\
&=\E\left[\frac{f'(F)\nabla_F^2}{\nabla_F^2+\varepsilon}\right]-\E\left[\frac{f(F)(\nabla_F^2-\varepsilon)}{(\nabla_F^2+\varepsilon)^2}\sum\limits_{k=1}^n  \partial_k\nabla_F \mathcal{L}_kg_k\right].
\end{align*}
The condition $\frac{\sum\limits_{k=1}^n g_k(X_k)}{\nabla_F},\frac{\sum\limits_{k=1}^n \partial_k\nabla_F\mathcal{L}_kg_k}{\nabla_F^2}\in L^1(\Omega)$ allows us to use the dominated convergence theorem, and by letting $\varepsilon\to 0,$ we get
$$\E[f'(F)]=\E\left[f(F)\left(\frac{\sum\limits_{k=1}^n g_k(X_k)}{\nabla_F}+\frac{\sum\limits_{k=1}^n \partial_k\nabla_F\mathcal{L}_kg_k}{\nabla_F^2}\right)\right].$$
The remaining of the proof is similar to {\it Step 2} in the proof of Proposition \ref{klf2}. We omit the details.
\end{proof}
Thanks to Proposition \ref{klf22} we deduce the following bound on the Fisher information distance.
\begin{thm}\label{8uh4}Let $F\in \mathcal{C}_{\mathbb{X}_n}^2\cap L^2(\Omega)$ be such that $\E[F]=0$ and ${\rm Var}(F)=1.$ For $p,q,r,s,t>1$ with $\frac{1}{p}+\frac{1}{q}+\frac{1}{r}=\frac{1}{s}+\frac{1}{t}=1$ and  for any continuous functions $g_k,1\leq k\leq n$ satisfying $\E|g_k(X_k)|<\infty,\E[g_k(X_k)]=0$ and $\nabla_F:=\sum\limits_{k=1}^n\partial_k F\mathcal{L}_kg_k\neq 0$ a.s. we have
\begin{align}
I(F\|Z)&\leq 3\big\|F-\sum\limits_{k=1}^n g_k(X_k)\big\|_2^2+3\big\|\sum\limits_{k=1}^n g_k(X_k)\big\|_{2p}^2\|\nabla_F^{-1}\|_{2q}^2 \|1-\nabla_F\|_{2r}^2\notag\\
&+3 \|\nabla_F^{-1}\|_{4s}^4\big\|\sum\limits_{k=1}^n \partial_k\nabla_F\mathcal{L}_kg_k\big\|_{2t}^2,\label{olfm3a}
\end{align}
provided that the norms in the right hand side are finite.
\end{thm}
\begin{proof} Using the definition (\ref{vail}) and the score function $\rho_F$ given by (\ref{d3t13b}) we obtain
\begin{align*}
I(F\|Z)&\leq \E\left|\frac{\sum\limits_{k=1}^n g_k(X_k)}{\nabla_F}+\frac{\sum\limits_{k=1}^n \partial_k\nabla_F\mathcal{L}_kg_k}{\nabla_F^2}-F\right|^2\\
&\leq \E\left|\frac{\sum\limits_{k=1}^n g_k(X_k)(1-\nabla_F)}{\nabla_F}+\frac{\sum\limits_{k=1}^n \partial_k\nabla_F\mathcal{L}_kg_k}{\nabla_F^2}+\sum\limits_{k=1}^n g_k(X_k)-F\right|^2.
\end{align*}
By the fundamental inequality $(a_1+a_2+a_3)^2\leq 3(a_1^2+a_2^2+a_3^2),$ we deduce
\begin{align}
I(F\|Z) &\leq 3\E\left|F-\sum\limits_{k=1}^n g_k(X_k)\right|^2\notag\\
&+ 3\E\left|\frac{\sum\limits_{k=1}^n g_k(X_k)(1-\nabla_F)}{\nabla_F}\right|^2+3\E\left|\frac{\sum\limits_{k=1}^n \partial_k\nabla_F\mathcal{L}_kg_k}{\nabla_F^2}\right|^2.\label{olfm3}
\end{align}
Furthermore, by using H\"older's inequality, we get the following estimates
$$
\E\left|\frac{\sum\limits_{k=1}^n g_k(X_k)(1-\nabla_F)}{\nabla_F}\right|^2\leq\big\|\sum\limits_{k=1}^n g_k(X_k)\big\|_{2p}^2\|\nabla_F^{-1}\|_{2q}^2 \|1-\nabla_F\|_{2r}^2
$$
and
\begin{align*}
\E\left|\frac{\sum\limits_{k=1}^n \partial_k\nabla_F\mathcal{L}_kg_k}{\nabla_F^2}\right|^2
&\leq \|\nabla_F^{-1}\|_{4s}^4\big\|\sum\limits_{k=1}^n \partial_k\nabla_F\mathcal{L}_kg_k\big\|_{2t}^2.
\end{align*}
Those estimates, together with (\ref{olfm3}), give us the desired bound (\ref{olfm3a}). The proof of the theorem is complete.
\end{proof}

\section{Applications}\label{89jk2}
In this section, we first apply our Fisher information bounds to two familiar classes of nonlinear statistics: $(i)$ Quadratic forms in Subsection \ref{89a} and $(ii)$ Functions of sample means in Subsection \ref{89b}. Although the central limit theorem for those statistics have been well studied, it remains an open problem to obtain quantitative bounds on the Fisher information distance. Then, in Subsection \ref{89c}, we discuss the Fisher information bounds for the sums of non-i.i.d random variables. Since our purpose is to illustrate the main features of the method developed in the previous section, we do not try to optimize the constants and moment conditions in our results. We simply chose to use $p=q=4$ and $r=s=t=2$ in Theorems \ref{gko} and \ref{8uh4}.%We also recall that  the random variables $X_k's$ always satisfy the assumptions required in Section \ref{klf5}.

%\subsection{Some useful moment inequalities}
It can be seen that the main tools in applying Theorems \ref{gko} and \ref{8uh4}  are the inequalities for negative and positive moments of nonlinear statistics. Several positive moment inequalities can be found in Chapter 15 of \cite{Boucheron2013}. In this section, we frequently use the following Marcinkiewicz-Zygmund type inequality.
\begin{prop} Let $\mathbb{Y}=(Y_1,\cdots,Y_n)$ be a vector of independent random variables. For any measurable function $U=U(\mathbb{Y})\in L^{p}(\Omega)$ for some $p\geq2,$ we have
\begin{equation}\label{9hj4}
\|U\|_p^2\leq |\E[U]|^2+ (p-1)\sum\limits_{k=1}^n\|U-\E_k[U]\|_p^2,
\end{equation}
where $\E_k$ denotes the expectation with respect to $Y_k.$%$\|.\|_p$ denotes the norm in $L^p(\Omega)$ and
\end{prop}
\begin{proof} We write $U=\E[U]+U_1+...+U_n,$ where $U_k=\E\left[U-\E_k[U]|Y_1,...,Y_k\right],\,\,1\leq k\leq n.$ Then, the inequality (\ref{9hj4}) follows directly from Theorem 2.1 in \cite{Rio2009a} and the fact that $\|U_k\|_p\leq \|U-\E_k[U]\|_p.$
\end{proof}
In order to bound the negative moments, we have the following.
\begin{prop}\label{ujmld}Let $(Y_k)_{k\geq 1}$ be a sequence of independent nonnegative random variables with $\E[Y_k]=1$ for all $k\geq 1.$ Assume that, for some $n_0\geq 1$ and $\alpha\geq 1,$ we have

\noindent (i) $\sup\limits_{k\geq n_0+1}\E|Y_k|^2<\infty,$

\noindent (ii) $\E[(Y_1 + \dots + Y_{n_0})^{-\alpha}]<\infty.$

\noindent Then, it holds that
$$\lim\limits_{n\to \infty}n^\alpha\E[(Y_1 + \dots + Y_{n})^{-\alpha}]=1.$$
\end{prop}
\begin{proof}For $0\leq k\leq n-1$ and $\varepsilon\in(0,\frac{1}{6}],$ we put $S_{k,n}:=Y_{k+1} + \cdots + Y_n,$ $S_{k,n,\varepsilon}:=Y_{k+1}\ind_{\{Y_{k+1}\leq (n-k)^\varepsilon\}} + \cdots + Y_n\ind_{\{Y_n\leq (n-k)^\varepsilon\}}$ and $\mu_{k,n,\varepsilon}:=\E[S_{k,n,\varepsilon}].$

We have
\begin{align*}
\E&[Y_{n_0+1}\ind_{\{Y_{n_0+1}> (n-n_0)^\varepsilon\}} ]+\cdots+\E[Y_n\ind_{\{Y_n> (n-n_0)^\varepsilon\}}]\\
&\leq\frac{\E[Y_{n_0+1}^2]+\cdots+\E[Y_n^2]}{(n-n_0)^\varepsilon}\\
&\leq\sup\limits_{k\geq n_0+1}\E[Y_k^2] (n-n_0)^{1-\varepsilon}\,\,\forall\,\,n\geq n_0+1,
\end{align*}
and hence,
\begin{equation*}
n-n_0-\sup\limits_{k\geq n_0+1}\E[Y_k^2](n-n_0)^{1-\varepsilon}\leq \mu_{n_0,n,\varepsilon}\leq n-n_0\,\,\forall\,\,n\geq n_0+1.
\end{equation*}
For $n>\bar{n}:=n_0+(2\sup\limits_{k\geq n_0+1}\E[Y_k^2])^{\frac{1}{\varepsilon}}>n_0+2^{\frac{1}{\varepsilon}},$ we have
\begin{equation}\label{iom4}
\frac{n-n_0}{2}\leq n-n_0-\sup\limits_{k\geq n_0+1}\E[Y_k^2](n-n_0)^{1-\varepsilon}\leq\mu_{n_0,n,\varepsilon}
\end{equation}
and
\begin{align}
\mu_{n_0,n,\varepsilon}-\mu_{n_0,n,\varepsilon}^{\frac{1}{2}+2\varepsilon}&\geq n-n_0-\sup\limits_{k\geq n_0+1}\E[Y_k^2](n-n_0)^{1-\varepsilon}-(n-n_0)^{\frac{1}{2}+2\varepsilon}\notag\\
&\geq n-n_0-2\sup\limits_{k\geq n_0+1}\E[Y_k^2](n-n_0)^{1-\varepsilon}>0.\label{iom4v}
\end{align}%Particularly, $\mu_{n_0,n,\varepsilon}-\mu_{n_0,n,\varepsilon}^{\frac{1}{2}+2\varepsilon}>0$ for all $\varepsilon\in(0,\frac{1}{4}).$
 We now observe that
\begin{align*}
&\E[S_{0,n}^{-\alpha}]\leq \E[(S_{0,n_0}+S_{n_0,n,\varepsilon})^{-\alpha}]\\
&=\E\left[(S_{0,n_0}+S_{n_0,n,\varepsilon})^{-\alpha}\ind_{\{S_{n_0,n,\varepsilon}\leq \mu_{n_0,n,\varepsilon}-\mu_{n_0,n,\varepsilon}^{\frac{1}{2}+2\varepsilon}\}}\right]
+\E\left[(S_{0,n_0}+S_{n_0,n,\varepsilon})^{-\alpha}\ind_{\{S_{n_0,n,\varepsilon}> \mu_{n_0,n,\varepsilon}-\mu_{n_0,n,\varepsilon}^{\frac{1}{2}+2\varepsilon}\}}\right]\\
&\leq \E[S_{0,n_0}^{-\alpha}]P(S_{n_0,n,\varepsilon}\leq\mu_{n_0,n,\varepsilon}-\mu_{n_0,n,\varepsilon}^{\frac{1}{2}+2\varepsilon})+
(\mu_{n_0,n,\varepsilon}-\mu_{n_0,n,\varepsilon}^{\frac{1}{2}+2\varepsilon})^{-\alpha},\,\,n>\bar{n}.
\end{align*}
Note that, for every $n_0+1\leq k\leq n,$ $Y_k\ind_{\{Y_k\leq (n-n_0)^\varepsilon\}}$ is a nonnegative random variable bounded by $(n-n_0)^\varepsilon.$ We therefore can use Hoeffding's inequality to get
\begin{align*}
P(S_{n_0,n,\varepsilon}\leq\mu_{n_0,n,\varepsilon}-\mu_{n_0,n,\varepsilon}^{\frac{1}{2}+2\varepsilon})
&=P(S_{n_0,n,\varepsilon}-\mu_{n_0,n,\varepsilon}\leq-\mu_{n_0,n,\varepsilon}^{\frac{1}{2}+2\varepsilon})\\
&\leq \exp\left(-\frac{2\mu_{n_0,n,\varepsilon}^{1+4\varepsilon}}{(n-n_0)^{1+2\varepsilon}}\right),\,\,n>\bar{n}.
\end{align*}
As a consequence,
$$\E[S_{0,n}^{-\alpha}]\leq\E[S_{0,n_0}^{-\alpha}]\exp\left(-\frac{2\mu_{n_0,n,\varepsilon}^{1+4\varepsilon}}{(n-n_0)^{1+2\varepsilon}}\right)+
(\mu_{n_0,n,\varepsilon}-\mu_{n_0,n,\varepsilon}^{\frac{1}{2}+2\varepsilon})^{-\alpha},\,\,\,n>\bar{n}.$$
This, combined with (\ref{iom4}) and (\ref{iom4v}), gives us
\begin{multline}\label{8usk}
n^\alpha\E[S_{0,n}^{-\alpha}]\leq\E[S_{0,n_0}^{-\alpha}]n^\alpha\exp\left(-\frac{2(n-n_0)^{2\varepsilon}}{2^{1+4\varepsilon}}\right)\\+
\bigg(\frac{n-n_0-2\sup\limits_{k\geq n_0+1}\E[Y_k^2](n-n_0)^{1-\varepsilon}}{n}\bigg)^{-\alpha},\,\,n>\bar{n}.
\end{multline} %Furthermore, it follows from (\ref{iom4}) that $\lim\limits_{n\to\infty}\frac{\mu_{n_0,n,\varepsilon}}{n}=1$ and $\lim\limits_{n\to\infty}\frac{\mu_{n_0,n,\varepsilon}^{\frac{1}{2}+2\varepsilon}}{n}=0.$
So we deduce
$$\limsup\limits_{n\to \infty}n^\alpha\E[S_{0,n}^{-\alpha}]\leq 1.$$
This finishes the proof because $n^\alpha\E[S_{0,n}^{-\alpha}]=n^\alpha\E[(Y_1 + \dots + Y_{n})^{-\alpha}]\geq 1$ by Jensen's inequality.

\end{proof}
\begin{rem}\label{ujiu} We recall that if $R$ is a non-negative random variable then its negative moments can be represented via the moment generating function:
\begin{equation}\label{iihg}
\E[R^{-\alpha}]=\frac{1}{\Gamma(\alpha)}\int_0^\infty x^{\alpha-1}\E[e^{-xR}]dx,\,\,\alpha>0,
\end{equation}
where $\Gamma(\alpha)=\int_0^\infty x^{\alpha-1}e^{-x}dx$ is the Gamma function. Thus we can check the condition $(ii)$ of Proposition \ref{ujmld} by using the following representation
$$\E[(Y_1 + \dots + Y_{n_0})^{-\alpha}]=\frac{1}{\Gamma(\alpha)}\int_0^\infty x^{\alpha-1}\prod\limits_{k=1}^{n_0}\E[e^{-xY_k}]dx.$$
For instance, if $\lim\limits_{x\to\infty}x^{\alpha_k}\E[e^{-xY_k}]$ is finite for some $\alpha_k>0$ then $\E|Y_1 + \dots + Y_{n_0}|^{-\alpha}<\infty$ for all $\alpha<\alpha_1+\cdots+\alpha_{n_0}.$
\end{rem}
\begin{rem}\label{u41u} We now choose to use $\varepsilon=\frac{1}{6}.$ Then, the relation (\ref{8usk}) gives us the following upper bound%it follows from the estimates (\ref{iom4}) and  (\ref{iom4v}) that
%$$\mu_{n_0,n,\varepsilon}-\mu_{n_0,n,\varepsilon}^{\frac{1}{2}+2\varepsilon}\geq \frac{n-n_0}{2}-(n-n_0)^{\frac{1}{2}+2\varepsilon}=\frac{n-n_0}{2}-(n-n_0)^{\frac{5}{6}}\geq 0$$
%for all $n>\bar{n}:=n_0+(2\sup\limits_{k\geq n_0+1}\E[Y_k^2])^{6}>n_0+64.$
\begin{align*}
n^\alpha&\E[(Y_1 + \dots + Y_{n})^{-\alpha}]=n^\alpha\E[S_{0,n}^{-\alpha}]\\
&\leq\E[(Y_1 + \dots + Y_{n_0})^{-\alpha}]n^\alpha\exp\left(-\frac{(n-n_0)^{\frac{1}{3}}}{2^{\frac{2}{3}}}\right)+
\bigg(\frac{n-n_0-2\sup\limits_{k\geq n_0+1}\E[Y_k^2](n-n_0)^{\frac{5}{6}}}{n}\bigg)^{-\alpha}
\end{align*}
for all $n>\bar{n}:=n_0+(2\sup\limits_{k\geq n_0+1}\E[Y_k^2])^{6}.$ On the other hand, we have
$$n^\alpha\E[(Y_1 + \dots + Y_{n})^{-\alpha}]\leq \bar{n}^\alpha \E[(Y_1 + \dots + Y_{n_0})^{-\alpha}],\,\,\,n_0\leq n\leq \bar{n}.$$
So it holds that
$$\sup\limits_{k\geq n_0}n^\alpha\E[(Y_1 + \dots + Y_{n})^{-\alpha}]\leq C_{\alpha,\bar{\beta}} (1+\E[(Y_1 + \dots + Y_{n_0})^{-\alpha}]),$$
where $C_{\alpha,\bar{\beta}}$ is a positive constant depending only on $\alpha$ and $\bar{\beta}:=\sup\limits_{k\geq n_0+1}\E|Y_k|^2.$
\end{rem}
\subsection{Quadratic forms}\label{89a}
Let  $A_n=(a^{(n)}_{uv})_{u,v=1}^n$  be a symmetric matrix with vanishing diagonal, where each $a^{(n)}_{uv}$ is a real number depending on $n.$ For the simplicity of notations, we will write $a_{uv}$ instead of $a^{(n)}_{uv}.$ The central limit theorem for the quadratic form
$$F_n=\sum\limits_{1\leq u\leq v\leq n} a_{uv}X_uX_v$$
has been extensively discussed in the literature. The best known result proved by de Jong \cite{deJong1987} tells us that the $\sigma_n^{-1}F_n$ converges to $Z\sim N(0,1)$  in distribution if
\begin{equation}\label{okf}
\text{$\sigma_n^{-4}\mathrm{Tr}(A_n^4)\to 0$ and $\sigma_n^{-2}\max\limits_{1\leq u\leq n}\sum\limits_{v=1}^na_{uv}^2\to 0$},\,\,n\to\infty,
\end{equation}
where $\sigma^2_n:={\rm Var}(F_n)=\sum\limits_{1\leq u\leq v\leq n} a_{uv}^2$ and $\mathrm{Tr}(A_n^4)=\sum\limits_{u,v=1}^n \big(\sum\limits_{k=1}^na_{ku} a_{kv}\big)^2.$ Recently, the following Berry-Esseen error bound has been  obtained in \cite{Nicolas2020,Shao2019}
\begin{multline}
\sup\limits_{x\in \mathbb{R}}\left|P(\sigma_n^{-1}F_n\leq x)-P(Z\leq x)\right|\\\leq \frac{C\sup\limits_{1\leq k\leq n}\E|X_k|^4}{\sigma_n^2}\left(\sqrt{\sum\limits_{k=1}^n\big(\sum\limits_{v=1}^na_{kv}^2\big)^2}+\sqrt{\sum\limits_{u,v=1}^n \big(\sum\limits_{k=1}^na_{ku} a_{kv}\big)^2}\right),\label{ohjm}
\end{multline}
where $C$ is an absolute constant.

Here we use Theorem \ref{gko} to bound $I(\sigma_n^{-1}F_n\|Z).$ The key point is to use the following differentiable decomposition of $F_n:$
$$\mathcal{M}_kF_n=\frac{1}{2}X_k\sum\limits_{v=1}^n a_{kv}X_v,\,\,1\leq k\leq n.$$
Notice that $a_{uu}=0$ for all $1\leq u\leq n.$ Then, we have $\mathcal{L}_k\mathcal{M}_kF_n=\frac{1}{2}\tau_k(X_k)\sum\limits_{v=1}^n a_{kv}X_v,\,\,1\leq k\leq n,$ and hence,
$$\Theta_{F_n}=\frac{1}{2}\sum\limits_{u=1}^n\big(\sum\limits_{v=1}^n a_{uv}X_v\big)^2\tau_u(X_u).$$
The next theorem describes our general result for the Fisher information distance of quadratic forms.
\begin{thm}\label{uij8}Assume that
$$\beta:=\max\big\{\sup\limits_{k\geq 1}\E|X_k|^8,\sup\limits_{k\geq 1}\E|\tau_k(X_k)|^8,\sup\limits_{k\geq 1}\E|\tau'_k(X_k)|^8\big\}<\infty.$$
%$$\beta:=\max\big\{\sup\limits_{k\geq 1}\E|X_k|^8,\sup\limits_{k\geq 1}\E|\tau_k(X_k)|^8,\sup\limits_{k\geq 1}\E|\tau'_k(X_k)|^6\big\}<\infty.$$
%(i) $\E|X_k|^8+\E[\tau_k(X_k)^4]+\E[|\tau'_k(X_k)|^4\tau_k(X_k)^2]<\infty$
%
%(ii) $\E|\Theta_{F_n}|^{-8}<\infty.$
%
\noindent Then, we have
\begin{equation}\label{kld1}
I(\sigma_n^{-1}F_n\|Z)\leq \frac{C_\beta}{\sigma_n^4}\left(\sum\limits_{k=1}^n\big(\sum\limits_{v=1}^na_{kv}^2\big)^2+\sum\limits_{u,v=1}^n \big(\sum\limits_{k=1}^na_{ku} a_{kv}\big)^2\right)\|\sigma_n^2\Theta_{F_n}^{-1}\|_8^3\,\,\forall\,n\geq 1,
\end{equation}
where $C_\beta$ is a positive constant depending only on $\beta.$
\end{thm}%
Obviously, we have
\begin{align*}
\frac{1}{\sigma_n^4}\left(\sum\limits_{k=1}^n\big(\sum\limits_{v=1}^na_{kv}^2\big)^2+\sum\limits_{u,v=1}^n \big(\sum\limits_{k=1}^na_{ku} a_{kv}\big)^2\right)&=\frac{1}{\sigma_n^4}\left(\sum\limits_{k=1}^n\big(\sum\limits_{v=1}^na_{kv}^2\big)^2+\mathrm{Tr}(A_n^4)\right)\\
&\leq \sigma_n^{-4}\mathrm{Tr}(A_n^4)+\sigma_n^{-2}\max\limits_{1\leq u\leq n}\sum\limits_{v=1}^na_{uv}^2.
\end{align*}
So we obtain the following central limit theorem result in information-theoretic sense.
\begin{cor}\label{df2}Assume that

\noindent (i) $\beta:=\max\big\{\sup\limits_{k\geq 1}\E|X_k|^8,\sup\limits_{k\geq 1}\E|\tau_k(X_k)|^8,\sup\limits_{k\geq 1}\E|\tau'_k(X_k)|^8\big\}<\infty,$

\noindent (ii) $\limsup\limits_{n\to\infty}\|\sigma_n^2\Theta_{F_n}^{-1}\|_8^3<\infty.$

\noindent  (iii) the condition (\ref{okf}) holds.

\noindent Then $\lim\limits_{n\to\infty}I(\sigma_n^{-1}F_n\|Z)=0.$
\end{cor}
The proof of Theorem \ref{uij8} will be given at the end of this subsection. We note that the error bound (\ref{kld1}) comes as no surprise: it can be guessed by using the Berry-Esseen bound (\ref{ohjm}) and the following known relation
$$\sup\limits_{x\in \mathbb{R}}\left|P(\sigma_n^{-1}F_n\leq x)-P(Z\leq x)\right|\leq \sqrt{I(\sigma_n^{-1}F_n\|Z)}.$$
The additional factor $\|\sigma_n^2\Theta_{F_n}^{-1}\|_8^3$ appearing in our bound characterizes the range of applicability of Theorem \ref{uij8}.  In other words, our Theorem \ref{uij8} only works for the class of random variables $X_k's$ satisfying $\|\sigma_n^2\Theta_{F_n}^{-1}\|_8^3<\infty.$ In the next Corollaries \ref{fkm9} and \ref{fdv}, we provide such two classes of random variables.%In order to make the bound (\ref{kld1}) more explicit, we need to find an upper bound on $\|\sigma_n^2\Theta_{F_n}^{-1}\|_8^3.$
\begin{cor}\label{fkm9}Assume that

\noindent (i) $\beta:=\max\big\{\sup\limits_{k\geq 1}\E|X_k|^8,\sup\limits_{k\geq 1}\E|\tau_k(X_k)|^8,\sup\limits_{k\geq 1}\E|\tau'_k(X_k)|^8\big\}<\infty,$

\noindent (ii) there exists  $\tau_0>0$ such that $\tau_k(X_k)\geq \tau_0\,\,a.s.$ for all $k\geq 1,$

\noindent  (iii) $\E|X_1^2+\cdots+X_{n_0}^2|^{-8}<\infty$ for some $n_0\geq 1.$

\noindent If $n\geq n_0$ and the matrix $A_n^TA_n$ is positive definite then  we have
\begin{equation}\label{hk1ld1}
I(\sigma_n^{-1}F_n\|Z)\leq \frac{C_{\beta,\tau_0}\gamma^{\frac{3}{8}}}{\sigma_n^4}\left(\sum\limits_{k=1}^n\big(\sum\limits_{v=1}^na_{kv}^2\big)^2+\sum\limits_{u,v=1}^n \big(\sum\limits_{k=1}^na_{ku} a_{kv}\big)^2\right)(\sigma_n^{-2}\lambda_{min}^{(n)} n)^{-3},
\end{equation}
where $\lambda_{min}^{(n)}$ is the smallest eigenvalue of $A_n^TA_n,$ $\gamma:=\sup\limits_{n\geq n_0}n^8\E|X_1^2+\cdots+X_n^2|^{-8}<\infty$ and $C_{\beta,\tau_0}$ is a positive constant depending only on $\beta,\tau_0.$
\end{cor}
\begin{proof}We put
$$\tilde{\Theta}_{F_n}:=\sum\limits_{u=1}^n\big(\sum\limits_{v=1}^n a_{uv}X_v\big)^2.$$
It can be seen that $\tilde{\Theta}_{F_n}$ is a quadratic form corresponding to the matrix $A_n^TA_n.$ Hence, we always have
 $$\tilde{\Theta}_{F_n}\geq \lambda_{min}^{(n)}(X_1^2+\cdots+ X_n^2)\,\,a.s.$$
By the assumption $\tau_k(X_k)\geq \tau_0>0$ for all $k\geq 1,$ we obtain
$$\Theta_{F_n}\geq \frac{\tau_0}{2}\tilde{\Theta}_{F_n}\geq\frac{\tau_0 \lambda_{min}^{(n)}}{2}(X_1^2+\cdots+ X_n^2)\,\,a.s.$$
Recalling Proposition \ref{ujmld}, we now use the conditions $(i)$ and $(iii)$ to get the fact that
$$\text{$\gamma:=\sup\limits_{n\geq n_0}n^8\E|X_1^2+\cdots+X_n^2|^{-8}$ is a finite constant},$$
and hence,
\begin{equation}\label{wmk1}
\|\sigma_n^2\Theta_{F_n}^{-1}\|_8\leq \frac{2\gamma\textcolor{blue}{^{\frac{1}{8}}}\sigma_n^2}{\tau_0\lambda_{min}^{(n)} n}<\infty,\,\,n\geq n_0.
\end{equation}
This estimate, together with the bound (\ref{kld1}), yields (\ref{hk1ld1}).
\end{proof}
\begin{rem}In the bound (\ref{hk1ld1}), the exact value of $\gamma$ is not easy to compute. However, in view of Remark \ref{u41u}, we can estimate $\gamma$ as follows
$$\gamma=\sup\limits_{k\geq n_0}n^8\E[(X_1^2 + \dots + X_{n}^2)^{-8}]\leq C_{\bar{\beta}} (1+\E[(X_1^2 + \dots + X_{n_0}^2)^{-8}]),$$
where $C_{\bar{\beta}}$ is a positive constant depending only on $\bar{\beta}:=\sup\limits_{k\geq n_0+1}\E|X_k|^4.$
\end{rem}
\begin{rem}From the bound (\ref{hk1ld1}), the rate of Fisher information convergence is consistent with the Berry-Esseen rate (\ref{ohjm}) if the matrix $A_n$ satisfies $\limsup\limits_{n\to\infty}(\sigma_n^{-2}\lambda_{min}^{(n)} n)^{-3}<\infty.$
\end{rem}
\begin{rem}We have $\E[\tilde{\Theta}_{F_n}]\geq \lambda_{min}^{(n)}\E[(X_1^2+\cdots+ X_n^2)]=\lambda_{min}^{(n)}n.$ Similarly, $\E[\tilde{\Theta}_{F_n}]\leq \lambda_{max}^{(n)}n,$ where $\lambda_{max}^{(n)}$ is the largest eigenvalue of $A_n^TA_n.$ Furthermore, $\E[\tilde{\Theta}_{F_n}]=2\E[\Theta_{F_n}]=2\sigma_n^2$ because $\E[\tau_k(X_k)]=1,k\geq 1.$ Thus the factor $(\sigma_n^{-2}\lambda_{min}^{(n)} n)^{-3}$ in the bound (\ref{hk1ld1}) can be bounded by
$$(\sigma_n^{-2}\lambda_{min}^{(n)} n)^{-3}\leq \frac{1}{8}(\lambda_{max}^{(n)}/\lambda_{min}^{(n)})^3.$$
 %Particularly, if $A_n$ is an orthonormal matrix then $\lambda_{max}^{(n)}=\lambda_{min}^{(n)}=1$ and the bound (\ref{hk1ld1}) becomes
%$$
%I(\sigma_n^{-1}F_n\|Z)\leq \frac{C_{\beta,\tau_0,\gamma}}{\sigma_n^4}\left(\sum\limits_{k=1}^n\big(\sum\limits_{v=1}^na_{kv}^2\big)^2+\sum\limits_{u,v=1}^n \big(\sum\limits_{k=1}^na_{ku} a_{kv}\big)^2\right).
%$$
For the reader's convenience we recall that
$$\lambda_{min}^{(n)}=\min\limits_{x\neq 0}\frac{\|A_nx\|^2}{\|x\|^2},\,\,\,\lambda_{max}^{(n)}=\max\limits_{x\neq 0}\frac{\|A_nx\|^2}{\|x\|^2},$$
where $\|.\|$ denotes Euclidean norm in $\mathbb{R}^n.$
\end{rem}
\begin{rem} The condition $(iii)$ of Corollary \ref{fkm9} can be checked by using the observation made in Remark \ref{ujiu}. Furthermore, we can use the characteristic function of $X_k$ to compute $\E[e^{-x X_k^2}],x\geq 0.$ Indeed, we have
$$\E[e^{-x X_k^2}]=\E[e^{i\sqrt{2x} X_k N}]=\E[\varphi_k(\sqrt{2x}N)],$$
where $i^2=-1,$ $\varphi_k$ denotes the characteristic function of $X_k$ and $N$ is standard normal random variable independent of $X_k.$
\end{rem}
\begin{rem}When $(X_k)_{k\geq 1}$ is the sequence of independent standard normal random variables, we can use Theorem 12 in \cite{Herry2023} to obtain a similar bound to (\ref{hk1ld1}). However, the method developed in \cite{Herry2023} requires the sample size $n$ to be sufficiently large. Here our method only require $n\geq n_0=17.$ Indeed, we have $\tau_k(X_k)=1$ and $\E[e^{-xX_k^2}]= (1+2 x)^{-\frac{1}{2}}$ for all $x\geq 0$ and $k\geq 1.$ Hence, all conditions of Corollary \ref{fkm9} are fulfilled ($\tau_0=1$ and $n_0=17$). Furthermore, standing for the estimate (\ref{wmk1}), we have the following exact expression
\begin{equation}\label{iihg2}
\|\sigma_n^2\Theta_{F_n}^{-1}\|_8=2\sigma_n^2\bigg(\frac{1}{\Gamma(8)}\int_0^\infty x^{7}\prod\limits_{k=1}^n  (1+2 \lambda_k x)^{-1/2}dx\bigg)^{1/8},
\end{equation}
where $\lambda_1,\cdots,\lambda_n$ are the eigenvalues of  $A_n^TA_n.$ Indeed, denoted by $\mathbb{V}_1,\cdots,\mathbb{V}_n$ the normalized eigenvectors of  $A_n^TA_n,$ we have
$$2\Theta_{F_n}=\tilde{\Theta}_{F_n}=\lambda_1Y_1^2+\cdots+\lambda_n Y_n^2,$$
where $Y_k:=\mathbb{V}_k{\mathbb{X}_n}^T,1\leq k\leq n.$ Note that $Y_k's$ are also independent standard normal random variables. So we obtain (\ref{iihg2}) by using the negative moment formula  (\ref{iihg}). %The reader can verify that $\lim\limits_{n\to\infty}\|\sigma_n^2\Theta_{F_n}^{-1}\|_8=1$ if $\lim\limits_{n\to\infty}\lambda_{max}^{(n)}/\sigma_n^2=0.$
\end{rem}
\begin{cor}\label{fdv}Assume that

\noindent (i) $\beta:=\max\big\{\sup\limits_{k\geq 1}\E|X_k|^8,\sup\limits_{k\geq 1}\E|\tau_k(X_k)|^8,\sup\limits_{k\geq 1}\E|\tau'_k(X_k)|^8\big\}<\infty,$

\noindent (ii) there exist  $0<\underline{\tau}_0\leq \overline{\tau}_0<\infty$ such that $\underline{\tau}_0\leq \tau_k(X_k)\leq \overline{\tau}_0\,\,a.s.$ for all $k\geq 1,$

\noindent If the largest eigenvalue $\lambda_{max}^{(n)}$ of $A_n^TA_n$ satisfies $\frac{\underline{\tau}_0\sigma_n^2}{\overline{\tau}_0\lambda_{max}^{(n)}}> 8$ then we have
%there exists a positive sequence $(c_n)_{n\geq n_0}$ such that, for all $n\geq n_0,$ the matrix $c_nA_n^TA_n-(A_n^TA_n)^2$ is non-negative definite and $\frac{\underline{\tau}_0\sigma_n^2}{c_n\overline{\tau}_0}\geq 8+\varepsilon$ for some $\varepsilon>0.$
\begin{equation}\label{fr5}
I(\sigma_n^{-1}F_n\|Z)\leq \frac{C_{\beta,\underline{\tau}_0}}{\sigma_n^4}\left(\sum\limits_{k=1}^n\big(\sum\limits_{v=1}^na_{kv}^2\big)^2+\sum\limits_{u,v=1}^n \big(\sum\limits_{k=1}^na_{ku} a_{kv}\big)^2\right),
\end{equation}
 where $C_{\beta,\underline{\tau}_0}$ is a positive constant depending only on $\beta$ and $\underline{\tau}_0.$
%$n_0\geq 1$ such that, for each $n\geq n_0,$ the matrix $c_nB_n^TB_n-B_n$  non-negative definite$B_n=:A_n^TA_n$
\end{cor}
\begin{proof}Let $\tilde{\Theta}_{F_n}$ be as in the proof of Corollary \ref{fkm9}. We write
$$\tilde{\Theta}_{F_n}:=\sum\limits_{u=1}^n\big(\sum\limits_{v=1}^n a_{uv}X_v\big)^2=\sum\limits_{u,v=1}^n b_{uv}X_uX_v,$$
where $(b_{uv})_{u,v=1}^n:=A_n^TA_n.$ Fixed $n\geq1,$ we consider the function
$$\varphi(x):=\E[e^{-x\tilde{\Theta}_{F_n}}],\,\,x\geq 0.$$
We have
\begin{align*}
\varphi'(x)&=-\E[\tilde{\Theta}_{F_n}e^{-x\tilde{\Theta}_{F_n}}]\\
&=-\sum\limits_{u=1}^n\E\left[X_u\sum\limits_{v=1}^n b_{uv}X_ve^{-x\tilde{\Theta}_{F_n}}\right]\\
&=-\sum\limits_{u=1}^n\E\left[\big(b_{uu}(X_u^2-\tau_u(X_u))+X_u\sum\limits_{v=1,v\neq u}^n b_{uv}X_v\big)e^{-x\tilde{\Theta}_{F_n}}\right]-\sum\limits_{u=1}^n\E[b_{uu}\tau_u(X_u)e^{-x\tilde{\Theta}_{F_n}}].
\end{align*}
For each $1\leq u\leq n,$ by using the formula (\ref{kflo}), we obtain
$$\E_u\left[\big(b_{uu}(X_u^2-\tau_u(X_u))+X_u\sum\limits_{v=1,v\neq u}^n b_{uv}X_v\big)e^{-x\tilde{\Theta}_{F_n}}\right]=-x\E_u\left[\big(\sum\limits_{v=1}^n b_{uv}X_v\big)^2\tau_u(X_u)e^{-x\tilde{\Theta}_{F_n}}\right],$$
and hence,
$$
\varphi'(x)=x\sum\limits_{u=1}^n\E\left[\big(\sum\limits_{v=1}^n b_{uv}X_v\big)^2\tau_u(X_u)e^{-x\tilde{\Theta}_{F_n}}\right]-\sum\limits_{u=1}^n\E[b_{uu}\tau_u(X_u)e^{-x\tilde{\Theta}_{F_n}}],\,\,x\geq 0.
$$
By the condition $(ii)$ we deduce
\begin{align*}
\varphi'(x)&\leq 2x\overline{\tau}_0\sum\limits_{u=1}^n\E\left[\big(\sum\limits_{v=1}^n b_{uv}X_v\big)^2e^{-x\tilde{\Theta}_{F_n}}\right]-\underline{\tau}_0\sigma_n^2\E[e^{-x\tilde{\Theta}_{F_n}}]\sum\limits_{u=1}^nb_{uu}\\
&= 2x\overline{\tau}_0\E\left[\sum\limits_{u=1}^n\big(\sum\limits_{v=1}^n b_{uv}X_v\big)^2e^{-x\tilde{\Theta}_{F_n}}\right]-2\underline{\tau}_0\sigma_n^2\E[e^{-x\tilde{\Theta}_{F_n}}],\,\,x\geq 0.
\end{align*}
Since the matrix $\lambda_{max}^{(n)}A_n^TA_n-(A_n^TA_n)^2$ is non-negative definite, this implies that $\sum\limits_{u=1}^n\big(\sum\limits_{v=1}^n b_{uv}X_v\big)^2\leq \lambda_{max}^{(n)}\tilde{\Theta}_{F_n}.$ We therefore obtain
\begin{align*}
\varphi'(x)&\leq 2\lambda_{max}^{(n)}\overline{\tau}_0x\E[\tilde{\Theta}_{F_n}e^{-x\tilde{\Theta}_{F_n}}]-2\underline{\tau}_0\sigma_n^2\E[e^{-x\tilde{\Theta}_{F_n}}]\\
&=-2\lambda_{max}^{(n)}\overline{\tau}_0x\varphi'(x)-2\underline{\tau}_0\sigma_n^2\varphi(x),\,\,x\geq 0.
\end{align*}
Solving the above differential inequality with the initial condition $\varphi(0)=1$ yields
$$\varphi(x)\leq (1+2\lambda_{max}^{(n)}\overline{\tau}_0x)^{-\frac{\underline{\tau}_0\sigma_n^2}{\lambda_{max}^{(n)}\overline{\tau}_0}},\,\,x\geq 0.$$
We now use the negative moment formula (\ref{iihg}) to get
\begin{align*}
\E[\tilde{\Theta}_{F_n}^{-8}]&=\frac{1}{\Gamma(8)}\int_0^\infty x^{7}\E[e^{-x\tilde{\Theta}_{F_n}}]dx\\
&\leq \frac{1}{\Gamma(8)}\int_0^\infty x^{7}(1+2\lambda_{max}^{(n)}\overline{\tau}_0x)^{-\frac{\underline{\tau}_0\sigma_n^2}{\lambda_{max}^{(n)}\overline{\tau}_0}}dx,
\end{align*}
which, by the substitution $y=\frac{1}{1+2c_n\overline{\tau}_0x},$ gives us
\begin{align*}
\E[\tilde{\Theta}_{F_n}^{-8}]&\leq  \frac{(2\lambda_{max}^{(n)}\overline{\tau}_0)^{-8}}{\Gamma(8)}\int_0^1 (1-y)^{7}y^{p-9}dy\,\,\,\text{(here $p:=\frac{\underline{\tau}_0\sigma_n^2}{\lambda_{max}^{(n)}\overline{\tau}_0}>8$)}\\
&=\frac{(2\lambda_{max}^{(n)}\overline{\tau}_0)^{-8}B(8,p-8)}{\Gamma(8)}\,\,\,\text{($B$ denotes the Beta function)}\\
&=\frac{(2\lambda_{max}^{(n)}\overline{\tau}_0)^{-8}\Gamma(p-8)}{\Gamma(p)}.
\end{align*}
It follows from the fundamental identity $\Gamma(x+1)=x\Gamma(x)$ that $\frac{\Gamma(p-8)}{\Gamma(p)}\leq \frac{8^7}{7!}p^{-8}.$ So we have
$$\E[\tilde{\Theta}_{F_n}^{-8}]\leq \frac{8^7}{7!}(2\lambda_{max}^{(n)}\overline{\tau}_0)^{-8}p^{-8}
=\frac{8^7}{7!}(2\underline{\tau}_0\sigma_n^2)^{-8}.$$
On the other hand, we have $\Theta_{F_n}\geq \frac{\underline{\tau}_0}{2}\tilde{\Theta}_{F_n}.$ So it holds that
\begin{equation}\label{8ui}
\|\sigma_n^2\Theta_{F_n}^{-1}\|_8\leq \bigg(\frac{8^7}{7!}\bigg)^{\frac{1}{8}}\underline{\tau}_0^{-2}.
\end{equation}
Inserting (\ref{8ui}) into the bound (\ref{kld1}), yields (\ref{fr5}). This completes the proof of the corollary.
\end{proof}
%\textcolor{blue}{\begin{exam}Let $\phi$ be the density of standard normal random variable and $\Phi(x)=\int_{-\infty}^x\phi(y)dy.$ $\beta:=\max\big\{\sup\limits_{k\geq 1}\E|X_k|^8,\sup\limits_{k\geq 1}\E|\tau_k(X_k)|^8,\sup\limits_{k\geq 1}\E|\tau'_k(X_k)|^8\big\}<\infty,$
%$$c\leq \frac{\phi(\Phi^{-1}(x))}{p_k(\nu_k^{-1}(x))}\leq C\,\,\forall\,x\in [0,1]$$
%for some $C\geq c>0,$ where $\nu_k$ denotes the cumulative distribution function of $X_k.$
%\end{exam}}
For the proof of Theorem \ref{uij8} we will need the following Lemmas \ref{ting1}, \ref{ting2} and \ref{ting3}.
\begin{lem}\label{ting1}Under assumption of Theorem \ref{uij8}, we have
$$\|\sigma_n^2-\Theta_{F_n}\|_4^2\leq C_\beta\left(\sum\limits_{k=1}^n\big(\sum\limits_{v=1}^na_{kv}^2\big)^2+\sum\limits_{u,v=1}^n \big(\sum\limits_{k=1}^na_{ku} a_{kv}\big)^2\right)\,\,\forall\,n\geq 1,$$
where $C_\beta$ is a positive constant depending only on $\beta.$
\end{lem}
\begin{proof}  By straightforward computations, we have
\begin{align*}
\Theta_{F_n}-\E_l \Theta_{F_n}&=\frac{1}{2}(\tau_l(X_l)-1)\big(\sum\limits_{u=1}^n a_{lu}X_u\big)^2\\
&+\frac{1}{2}(X^2_l-1)\sum\limits_{k=1}^n a^{2}_{kl}\tau_k(X_k)+X_l\sum\limits_{v=1,v\neq l}^n\big(\sum\limits_{k=1}^na_{kl} a_{kv}\tau_k(X_k)\big)X_v,\,\,1\leq l\leq n
\end{align*}
 (here we recall that $a_{uu}=0$ for all $1\leq u\leq n.$) Hence, by the triangle inequality, we obtain
\begin{align}
\|\Theta_{F_n}-\E_l \Theta_{F_n}\|_4\leq \frac{1}{2}\|\tau_l(X_l)-&1\|_4\big\|\sum\limits_{u=1}^n a_{lu}X_u\big\|_8^2+\frac{1}{2} \|X^2_l-1\|_4\big\|\sum\limits_{k=1}^n a^{2}_{kl}\tau_k(X_k)\big\|_4\notag\\
&+\|X_l\|_4\big\|\sum\limits_{v=1,v\neq l}^n\big(\sum\limits_{k=1}^na_{kl} a_{kv}\tau_k(X_k)\big)X_v\big\|_4,\,\,1\leq l\leq n.\label{bgj3}
\end{align}
By the inequality (\ref{9hj4})  we deduce
\begin{equation}\label{913a}
\big\|\sum\limits_{u=1}^n a_{lu}X_u\big\|_8^{2}\leq 7\max\limits_{1\leq u\leq n}\|X_u\|_8^{2}\sum\limits_{u=1}^n a^2_{lu},\,\,1\leq l\leq n
\end{equation}
and by the triangle inequality
\begin{equation}\label{913b}
\big\|\sum\limits_{k=1}^n a^{2}_{kl}\tau_k(X_k)\big\|_4\leq \sum\limits_{k=1}^n\big\| a^{2}_{kl}\tau_k(X_k)\big\|_4\leq \max\limits_{1\leq k\leq n}\|\tau_k(X_k)\|_4\sum\limits_{k=1}^n a^2_{kl},\,\,1\leq l\leq n.
\end{equation}
To estimate the third addend in the right hand side of (\ref{bgj3}), we put
$$W_l:=\sum\limits_{v=1,v\neq l}^n\big(\sum\limits_{k=1}^na_{kl} a_{kv}\tau_k(X_k)\big)X_v,\,\,1\leq l\leq n.$$
We have
$$W_l-\E_wW_l=\big(\sum\limits_{k=1}^na_{kl} a_{kw}\tau_k(X_k)\big)X_w+\sum\limits_{v=1,v\neq l,w}^na_{wl} a_{wv}(\tau_w(X_w)-1)X_v,\,\,1\leq w\leq n,$$
and hence,
\begin{align*}
\|W_l-\E_wW_l\|_4^2&\leq \bigg(\big\|\big(\sum\limits_{k=1}^na_{kl} a_{kw}\tau_k(X_k)\big)X_w\big\|_4+\big\|\sum\limits_{v=1,v\neq l,w}^na_{wl} a_{wv}(\tau_w(X_w)-1)X_v\big\|_4\bigg)^2\\
&=\bigg(\big\|\sum\limits_{k=1}^na_{kl} a_{kw}\tau_k(X_k)\big\|_4\|X_w\|_4+\big\|\sum\limits_{v=1,v\neq l,w}^na_{wl} a_{wv}X_v\big\|_4\|\tau_w(X_w)-1\|_4\bigg)^2\\
&\leq 2\max\limits_{1\leq w\leq n}\|X_w\|_4^{2}\big\|\sum\limits_{k=1}^na_{kl} a_{kw}\tau_k(X_k)\big\|_4^{2}\\
&+2\max\limits_{1\leq w\leq n}\|\tau_w(X_w)-1\|_4^{2}\big\|\sum\limits_{v=1,v\neq l,w}^na_{wl} a_{wv}X_v\big\|_4^{2},\,\,1\leq l,w\leq n.
\end{align*}
Once again, we use the inequality (\ref{9hj4}) to get
\begin{align*}
\|W_l-\E_wW_l\|_4^2&\leq 2\max\limits_{1\leq w\leq n}\|X_w\|_4^{2}\bigg(\big(\sum\limits_{k=1}^na_{kl} a_{kw}\big)^{2}+3\max\limits_{1\leq k\leq n}\|\tau_k(X_k)-1\|_4^{2}\sum\limits_{k=1}^n|a_{kl} a_{kw}|^{2}\bigg)\\
&+6\max\limits_{1\leq w\leq n}\|\tau_w(X_w)-1\|_4^{2}\max\limits_{1\leq w\leq n}\|X_w\|_4^{2}\sum\limits_{v=1}^n |a_{wl} a_{wv}|^{2}\\
&\leq 2\max\limits_{1\leq w\leq n}\|X_w\|_4^{2}\bigg(\big(\sum\limits_{k=1}^na_{kl} a_{kw}\big)^{2}+6\max\limits_{1\leq k\leq n}\|\tau_k(X_k)-1\|_4^{2}\sum\limits_{k=1}^n|a_{kl} a_{kw}|^{2}\bigg).
\end{align*}
Hence, we obtain
\begin{align}
\big\|&\sum\limits_{v=1,v\neq l}^n\big(\sum\limits_{k=1}^na_{kl}a_{kv}\tau_k(X_k)\big)X_v\big\|_4^2\notag\\
&=\|W_l\|_{4}^2\leq 3\sum\limits_{w=1}^n\|W_l-\E_wW_l|\|_4^{2}\notag\\
&\leq \max\limits_{1\leq w\leq n}\|X_w\|_4^{2}\bigg(\sum\limits_{w=1}^n\big(\sum\limits_{k=1}^na_{kl} a_{kw}\big)^{2}+6\max\limits_{1\leq k\leq n}\|\tau_k(X_k)-1\|_4^{2}\sum\limits_{w=1}^n\sum\limits_{k=1}^n|a_{kl} a_{kw}|^{2}\bigg).\label{r913b}
\end{align}
Inserting the estimates (\ref{913a}), (\ref{913b}) and (\ref{r913b}) into (\ref{bgj3}) yields
\begin{align}
\|&\Theta_{F_n}-\E_l \Theta_{F_n}\|_4^2\notag\\
&\leq \frac{3}{4}\|\tau_l(X_l)-1\|_4^2\big\|\sum\limits_{u=1}^n a_{lu}X_u\big\|_8^4+\frac{3}{4} \|X^2_l-1\|_4^2\big\|\sum\limits_{k=1}^n a^{2}_{kl}\tau_k(X_k)\big\|_4^2\notag\\
&+3\|X_l\|_4^2\big\|\sum\limits_{v=1,v\neq l}^n\big(\sum\limits_{k=1}^na_{kl} a_{kv}\tau_k(X_k)\big)X_v\big\|_4^2\notag\\
&\leq \frac{3\times49}{4}\|\tau_l(X_l)-1\|_4^2\max\limits_{1\leq u\leq n}\|X_u\|_8^{4}\big(\sum\limits_{u=1}^n a^2_{lu}\big)^2+\frac{3}{4} \|X^2_l-1\|_4^2\max\limits_{1\leq k\leq n}\|\tau_k(X_k)\|_4^2\big(\sum\limits_{k=1}^n a^2_{kl}\big)^2\notag\\
&+3\max\limits_{1\leq w\leq n}\|X_w\|_4^{4}\bigg(\sum\limits_{w=1}^n\big(\sum\limits_{k=1}^na_{kl} a_{kw}\big)^{2}+6\max\limits_{1\leq k\leq n}\|\tau_k(X_k)-1\|_4^{2}\sum\limits_{w=1}^n\sum\limits_{k=1}^n|a_{kl} a_{kw}|^{2}\bigg).\notag
\end{align}
Consequently, for some $C_\beta>0$ depending only on $\beta,$ we obtain
\begin{align}
\|\sigma_n^2-\Theta_{F_n}\|_4^2&\leq 3\sum\limits_{l=1}^n\|\Theta_{F_n}-\E_l \Theta_{F_n}\|_4^2\notag\\
&\leq C_\beta\left(\sum\limits_{k=1}^n\big(\sum\limits_{v=1}^na_{kv}^2\big)^2+\sum\limits_{u,v=1}^n \big(\sum\limits_{k=1}^na_{ku} a_{kv}\big)^2\right).\notag
\end{align}
This completes the proof of the lemma.
\end{proof}
\begin{lem}\label{ting2}Under the assumptions of Theorem \ref{uij8}, we have
$$\big\|\sum\limits_{k=1}^n|\partial_k\Theta_F|^2\tau_k(X_k)\big\|_2\leq C_\beta\left(\sum\limits_{k=1}^n\big(\sum\limits_{v=1}^na_{kv}^2\big)^2+\sum\limits_{u,v=1}^n \big(\sum\limits_{k=1}^na_{ku} a_{kv}\big)^2\right)\,\,\forall\,n\geq 1,$$
where $C_\beta$ is a positive constant depending only on $\beta.$
\end{lem}
\begin{proof}We have
\begin{align*}
\partial_k\Theta_{F_n}&=\sum\limits_{u=1}^na_{uk}\tau_u(X_u)\sum\limits_{v=1}^n a_{uv}X_v+\frac{1}{2}\tau'_k(X_k)\big(\sum\limits_{v=1}^n a_{kv}X_v\big)^2\\
&=\sum\limits_{v=1}^n\big(\sum\limits_{u=1}^na_{uk}a_{uv}\tau_u(X_u)\big)X_v+\frac{1}{2}\tau'_k(X_k)\big(\sum\limits_{v=1}^n a_{kv}X_v\big)^2,\,\,1\leq k\leq n.
\end{align*}
Hence, we deduce
$$|\partial_k\Theta_{F_n}|^2\leq2\bigg(\sum\limits_{v=1}^n\big(\sum\limits_{u=1}^na_{uk}a_{uv}\tau_u(X_u)\big)X_v\bigg)^2+\frac{1}{2}|\tau'_k(X_k)|^2
\big(\sum\limits_{v=1}^n a_{kv}X_v\big)^4,\,\,1\leq k\leq n.$$
We put
$A_k:=\sum\limits_{v=1}^n\big(\sum\limits_{u=1}^na_{uk}a_{uv}\tau_u(X_u)\big)X_v$ for $1\leq k\leq n,$ and set
$$A:=\sum\limits_{k=1}^nA_k^2\tau_k(X_k),\,\,\,B:=\sum\limits_{k=1}^n|\tau'_k(X_k)|^2\tau_k(X_k)\big(\sum\limits_{v=1}^n a_{kv}X_v\big)^4.$$
Then, we obtain
$\sum\limits_{k=1}^n|\partial_k\Theta|^2\tau_k(X_k)\leq 2A+\frac{1}{2}B$
and
\begin{equation}\label{kkgg1}
\big\|\sum\limits_{k=1}^n|\partial_k\Theta|^2\tau_k(X_k)\big\|_2\leq 2\|A\|_2+\frac{1}{2}\|B\|_2.
\end{equation}
By using the triangle inequality and the inequality  (\ref{9hj4}) we obtain
\begin{align}
\|B\|_2&\leq \sum\limits_{k=1}^n\||\tau'_k(X_k)|^2\tau_k(X_k)\|_2\big\|\sum\limits_{v=1}^n a_{kv}X_v\big\|_8^4\notag\\
&\leq\max\limits_{1\leq k\leq n} \||\tau'_k(X_k)|^2\tau_k(X_k)\|_2\sum\limits_{k=1}^n\big\|\sum\limits_{v=1}^n a_{kv}X_v\big\|_8^4\notag\\
&\leq 49\max\limits_{1\leq k\leq n} \||\tau'_k(X_k)|^2\tau_k(X_k)\|_2\max\limits_{1\leq v\leq n}\|X_v\|_8^{4}\sum\limits_{k=1}^n\big(\sum\limits_{v=1}^n a_{kv}^2\big)^2.\label{kkgg2}
\end{align}
%$$\E_lA_k:=\sum\limits_{v=1,v\neq l}^n\big(\sum\limits_{u=1,u\neq l}^na_{uk}a_{uv}\tau_u(X_u)+a_{lk}a_{lv}\big)X_v$$
We now observe that, for each $1\leq k\leq n,$ $\E A_k=0$ and
$$A_k-\E_lA_k=(\tau_l(X_l)-1)\sum\limits_{v=1,v\neq l}^na_{lk}a_{lv}X_v+\big(\sum\limits_{u=1}^na_{uk}a_{ul}\tau_u(X_u)\big)X_l,\,\,1\leq l\leq n.$$
With several applications of  the inequality (\ref{9hj4}), we deduce
\begin{align*}
\|&A_k\|_8^2\leq 7\sum\limits_{l=1}^n\|A_k-\E_lA_k\|_8^2\\
&\leq 7\sum\limits_{l=1}^n\bigg(\|\tau_l(X_l)-1\|_8\big\|\sum\limits_{v=1,v\neq l}^na_{lk}a_{lv}X_v\big\|_8+\big\|\sum\limits_{u=1}^na_{uk}a_{ul}\tau_u(X_u)\big\|_8\|X_l\|_8\bigg)^2\\
&\leq 14\max\limits_{1\leq l\leq n}\|\tau_l(X_l)-1\|_8^2\sum\limits_{l=1}^n\big\|\sum\limits_{v=1,v\neq l}^na_{lk}a_{lv}X_v\big\|_8^2+ 14\max\limits_{1\leq l\leq n}\|X_l\|_8^2\sum\limits_{l=1}^n\big\|\sum\limits_{u=1}^na_{uk}a_{ul}\tau_u(X_u)\big\|_8^2\\
&\leq 14\max\limits_{1\leq l\leq n}\|\tau_l(X_l)-1\|_8^2\max\limits_{1\leq l\leq n}\|X_l\|_8^2\sum\limits_{l=1}^n\sum\limits_{v=1,v\neq l}^n|a_{lk}a_{lv}|^2\\
&+ 14\max\limits_{1\leq l\leq n}\|X_l\|_8^2\sum\limits_{l=1}^n\bigg(\big(\sum\limits_{u=1}^na_{uk}a_{ul}\big)^2+\max\limits_{1\leq l\leq n}\|\tau_u(X_u)-1\|_8^2\sum\limits_{u=1}^n|a_{uk}a_{ul}|^2\bigg)\\
&\leq  14\max\limits_{1\leq l\leq n}\|X_l\|_8^2\sum\limits_{l=1}^n\big(\sum\limits_{u=1}^na_{uk}a_{ul}\big)^2+ 28\max\limits_{1\leq l\leq n}\|\tau_l(X_l)-1\|_8^2\max\limits_{1\leq l\leq n}\|X_l\|_8^2\sum\limits_{l=1}^n\sum\limits_{v=1}^n|a_{lk}a_{lv}|^2.
\end{align*}
Note that, by the triangle inequality, we have
$$\|A\|_2\leq\sum\limits_{k=1}^n\|A_k^2\tau_k(X_k)\|_2\leq\sum\limits_{k=1}^n\|A_k\|_8^2\|\tau_k(X_k)\|_4\leq\max\limits_{1\leq k\leq n}\|\tau_k(X_k)\|_4\sum\limits_{k=1}^n\|A_k\|_8^2.$$
As a consequence, we deduce
\begin{multline}\label{kkgg3}
\|A\|_2\leq\max\limits_{1\leq k\leq n}\|\tau_k(X_k)\|_4\bigg(14\max\limits_{1\leq l\leq n}\|X_l\|_8^2\sum\limits_{u,v=1}^n \big(\sum\limits_{k=1}^na_{ku} a_{kv}\big)^2\\
+ 28\max\limits_{1\leq l\leq n}\|\tau_l(X_l)-1\|_8^2\max\limits_{1\leq l\leq n}\|X_l\|_8^2\sum\limits_{k=1}^n\big(\sum\limits_{v=1}^n a_{kv}^2\big)^2\bigg).
\end{multline}
So we get the desired estimate by inserting (\ref{kkgg2}) and (\ref{kkgg3})  into (\ref{kkgg1}).
\end{proof}

\begin{lem}\label{ting3}Under assumption of Theorem \ref{uij8}, we have
\begin{equation}\label{klf33}
\|F_n\|_8^2\leq C_\beta\sigma_n^2\,\,\forall\,n\geq 1
\end{equation}
and
\begin{equation}\label{klf39}
\bigg\|\frac{\sum\limits_{k=1}^n\frac{|\mathcal{L}_k\mathcal{M}_kF_n|^2}{\tau_k(X_k)}}{\Theta_{F_n}^4}\bigg\|_2\leq\frac{1}{4}\|\Theta_{F_n}^{-1}\|_8^3\,\,\forall\,n\geq 1,
\end{equation}
where $C_\beta$ is a positive constant depending only on $\beta.$
\end{lem}
\begin{proof} We have
$$F_n-\E_kF_n=X_k\sum\limits_{v=1}^n a_{kv}X_v,\,\,1\leq k\leq n.$$
By using the inequality (\ref{9hj4}), we obtain
\begin{align*}
\|F_n\|_8^2&\leq 7\sum\limits_{k=1}^n\big\|X_k\sum\limits_{v=1}^n a_{kv}X_v\big\|_8^2\\
&= 7\sum\limits_{k=1}^n\|X_k\|_8^2\big\|\sum\limits_{v=1}^n a_{kv}X_v\big\|_8^2\\
&\leq 49\max\limits_{1\leq k\leq n}\|X_k\|_8^4\sum\limits_{k=1}^n\sum\limits_{v=1}^n a^2_{kv}\\
&=98\max\limits_{1\leq k\leq n}\|X_k\|_8^4\sigma_n^2.
\end{align*}
So (\ref{klf33}) follows. In order to check (\ref{klf39}), we recall that
$$\mathcal{L}_k\mathcal{M}_kF_n=\frac{1}{2}\tau_k(X_k)\sum\limits_{v=1}^n a_{kv}X_v,\,\,1\leq k\leq n.$$
We therefore obtain
$$\bigg\|\frac{\sum\limits_{k=1}^n\frac{|\mathcal{L}_k\mathcal{M}_kF_n|^2}{\tau_k(X_k)}}{\Theta_{F_n}^4}\bigg\|_2=\frac{1}{4}\|\Theta_{F_n}^{-1}\|_6^3\leq \frac{1}{4}\|\Theta_{F_n}^{-1}\|_8^3.$$
The proof of the proposition is complete.
\end{proof}

\noindent{\it Proof of Theorem \ref{uij8}.} Applying Theorem \ref{gko} to $F=\sigma_n^{-1}F_n$ and $p=q=4,r=s=t=2,$ we obtain
$$
I(\sigma_n^{-1}F_n\|Z)\leq \frac{2}{\sigma_n^2}\|F_n\|_8^2\|\Theta_{F_n}^{-1}\|_8^2\|\sigma_n^2-\Theta_{F_n}\|_4^2
+2\sigma_n^2\bigg\|\frac{\sum\limits_{k=1}^n\frac{|\mathcal{L}_k\mathcal{M}_k{F_n}|^2}{\tau_k(X_k)}}{\Theta_{F_n}^4}\bigg\|_2
\big\|\sum\limits_{k=1}^n|\partial_k\Theta_{F_n}|^2\tau_k(X_k)\big\|_2.
$$
This, combined with the estimates obtained in Lemmas \ref{ting1}, \ref{ting2} and \ref{ting3}, gives us
$$I(\sigma_n^{-1}F_n\|Z)\leq \frac{C_\beta}{\sigma_n^4}\left(\sum\limits_{k=1}^n\big(\sum\limits_{v=1}^na_{kv}^2\big)^2+\sum\limits_{u,v=1}^n \big(\sum\limits_{k=1}^na_{ku} a_{kv}\big)^2\right)(\|\sigma_n^2\Theta_{F_n}^{-1}\|_8^2+\|\sigma_n^2\Theta_{F_n}^{-1}\|_8^3)$$
for some $C_\beta>0.$ We now note that $\E[\Theta_{F_n}]=\sigma_n^2,$ and hence, by Jensen's inequality, $\|\sigma_n^2\Theta_{F_n}^{-1}\|_8\geq 1.$ So we get the desired bound (\ref{kld1}). The proof of Theorem \ref{uij8} is complete.
\subsection{Functions of sample means}\label{89b}
Let $H:\mathbb{R}\to \mathbb{R}$ be a twice differentiable function with bounded derivatives. We consider the statistic
\begin{equation}\label{i4}
F_n:=\sqrt{n}\left(H(\bar{X})-\E[H(\bar{X})]\right),\,\,n\geq 1,
\end{equation}
where $\bar{X}:=\frac{X_1 + \dots + X_n}{n}.$ This statistic is a generalization of the normalized sums discussed in Introduction, $F_n$ reduces to $Z_n$ when $H(x)=x$ and $X_k's$ are i.i.d random variables. The central limit theorem for $F_n$ and the Berry-Esseen error bound with rate $O(n^{-1/2})$ have been well studied, see e.g. \cite{Bentkus1997}.  Here we use Theorem \ref{8uh4} to estimate the Fisher information distance. %We expect, under suitable assumptions, to obtain the following rate of convergence
%$$I(\sigma_n^{-1}F_n\|Z)\leq \frac{C}{n},$$
%where $\sigma_n^2:={\rm Var}(F_n).$
Before stating and proving the main result of this subsection, we prepare some technical lemmas. In the proofs, we frequently use the following estimate
$$\E|\bar{X}|^p\leq (p-1)^{\frac{p}{2}}n^{-\frac{p}{2}}\max\limits_{1\leq k\leq n}\E|X_k|^p,\,\,p\geq 2,$$
which is a direct consequence of the Marcinkiewicz-Zygmund inequality (\ref{9hj4}).
\begin{lem}\label{re1} Assume that $\E|X_k|^4<\infty$ for all $k\geq 1.$ Then, for all $n\geq 1,$ we have
%$$\sigma_n^2:={\rm Var}(F_n)$$
\begin{equation}\label{old1}
|{\rm Var}(F_n)-|H'(0)|^2|\leq \frac{C_H}{\sqrt{n}}\max\limits_{1\leq k\leq n}\E|X_k|^4
\end{equation}
and
\begin{equation}\label{old2}
\big\|F_n-\sqrt{n}H'(0)\bar{X}\big\|_2^2\leq \frac{C_H}{n}\max\limits_{1\leq k\leq n}\E|X_k|^4,
\end{equation}
where $C_H$ is a positive constant depending only on $\|H'\|_\infty$ and $\|H''\|_\infty.$
\end{lem}
\begin{proof} Using the Taylor expansion of $H(x)$ at $x=0$ we obtain
$$H(\bar{X})=H(0)+H'(0)\bar{X}+\frac{1}{2}H''(\zeta)\bar{X}^2$$
for some random variable $\zeta.$ Hence, we deduce
$$F_n=\sqrt{n}H'(0)\bar{X}+\frac{\sqrt{n}}{2}H''(\zeta)\bar{X}^2-\frac{\sqrt{n}}{2}\E[H''(\zeta)\bar{X}^2]$$
and
$$
{\rm Var}(F_n)=|H'(0)|^2+nH'(0)\E[H''(\zeta)\bar{X}^3]+\frac{n}{4}{\rm Var}(H''(\zeta)\bar{X}^2).
$$
Consequently,
\begin{align*}
|{\rm Var}(F_n)-|H'(0)|^2|&\leq n\|H'\|_\infty\|H''\|_\infty\E|\bar{X}|^3+\frac{n\|H''\|_\infty^2}{4}\E|\bar{X}|^4\\
&\leq \frac{2^{3/2}\|H'\|_\infty\|H''\|_\infty}{\sqrt{n}}\max\limits_{1\leq k\leq n}\E|X_k|^3+\frac{9\|H''\|_\infty^2}{4n}\max\limits_{1\leq k\leq n}\E|X_k|^4.
\end{align*}
So (\ref{old1}) is proved. On the other hand, we have
\begin{align*}
\big\|F_n-\sqrt{n}H'(0)\bar{X}\big\|_2^2&=\frac{n}{4}\E|H''(\zeta)\bar{X}^2-\E[H''(\zeta)\bar{X}^2]|^2\\
&\leq \frac{n\|H''\|_\infty^2}{4}\E|\bar{X}|^4\\
&\leq \frac{9\|H''\|_\infty^2}{4n}\max\limits_{1\leq k\leq n}\E|X_k|^4,
\end{align*}
and we obtain (\ref{old2}). The proof of the lemma is complete.
\end{proof}
\begin{lem}\label{re2} Define
$$\nabla_{F_n}:=\frac{H'(0)H'(\bar{X})}{n}\sum\limits_{k=1}^n\tau_k(X_k).$$
Assume that $\E|X_k|^8+\E|\tau_k(X_k)|^8<\infty$ for all $k\geq 1.$ Then, for all $n\geq 1,$  we have
\begin{equation}\label{rold3}
\|{\rm Var}(F_n)-\nabla_{F_n}\|_4^2\leq\frac{C_H}{n}\max\limits_{1\leq k\leq n}\|X_k\|_8^4\max\limits_{1\leq k\leq n}\|\tau_k(X_k)\|_8^2,
\end{equation}
where $C_H$ is a positive constant depending only on $\|H'\|_\infty$ and $\|H''\|_\infty.$
\end{lem}
\begin{proof}We have
\begin{align}
\nabla_{F_n}&=\frac{|H'(0)|^2}{n}\sum\limits_{k=1}^n\tau_k(X_k)+\frac{H'(0)(H'(\bar{X})-H'(0))}{n}\sum\limits_{k=1}^n\tau_k(X_k).\notag
\end{align}
Hence, by the Taylor expansion, we deduce
$$\nabla_{F_n}=\frac{|H'(0)|^2}{n}\sum\limits_{k=1}^n\tau_k(X_k)+\frac{H'(0)H''(\eta)}{n}\bar{X}\sum\limits_{k=1}^n\tau_k(X_k)$$
for some random variable $\eta.$ We obtain
%$${\rm Var}(F_n)=|H'(0)|^2+nH'(0)\E[H''(\zeta)\bar{X}^3]+\frac{n}{4}{\rm Var}(H''(\zeta)\bar{X}^2).$$
\begin{align*}
\nabla_{F_n}&-{\rm Var}(F_n)=\frac{|H'(0)|^2}{n}\sum\limits_{k=1}^n(\tau_k(X_k)-1)\\
&+\frac{H'(0)H''(\eta)}{n}\bar{X}\sum\limits_{k=1}^n\tau_k(X_k)
-{\rm Var}(F_n)+|H'(0)|^2,
\end{align*}
and hence,
\begin{align}
\|\nabla_{F_n}&-{\rm Var}(F_n)\|_4^2\leq\frac{3|H'(0)|^4}{n^2}\big\|\sum\limits_{k=1}^n(\tau_k(X_k)-1)\big\|_4^2\notag\\
&+\frac{3\|H'\|_\infty^2\|H''\|_\infty^2}{n^2}\big\|\bar{X}\sum\limits_{k=1}^n\tau_k(X_k)\big\|_4^2+3|{\rm Var}(F_n)-|H'(0)|^2|^2.\label{old1a3}
\end{align}
By using the inequality (\ref{9hj4}) and the Cauchy-Schwarz inequality we get
\begin{equation}\label{old1a1}
\big\|\sum\limits_{k=1}^n(\tau_k(X_k)-1)\big\|_4^2\leq 3\sum\limits_{k=1}^n\|\tau_k(X_k)-1\|_4^2\leq 3n \max\limits_{1\leq k\leq n}\|\tau_k(X_k)-1\|_4^2
\end{equation}
and
\begin{align}
\big\|\bar{X}\sum\limits_{k=1}^n\tau_k(X_k)\big\|_4^2&\leq \|\bar{X}\|_8^2\big\|\sum\limits_{k=1}^n\tau_k(X_k)\big\|_8^2\notag\\
&\leq \frac{7}{n^2}\sum\limits_{k=1}^n\|X_k\|_8^2\big(\sum\limits_{k=1}^n\|\tau_k(X_k)\|_8\big)^2\notag\\
&\leq 7n \max\limits_{1\leq k\leq n}\|X_k\|_8^2\max\limits_{1\leq k\leq n}\|\tau_k(X_k)\|_8^2.\label{old1a2}
\end{align}
Inserting the estimates (\ref{old1a1}) and (\ref{old1a2}) into (\ref{old1a3}) gives us
\begin{align*}
\|\nabla_{F_n}-&{\rm Var}(F_n)\|_4^2\leq\frac{9|H'(0)|^4}{n}\max\limits_{1\leq k\leq n}\|\tau_k(X_k)-1\|_4^2\\
&+\frac{21\|H'\|_\infty^2\|H''\|_\infty^2}{n}\max\limits_{1\leq k\leq n}\|X_k\|_8^2\max\limits_{1\leq k\leq n}\|\tau_k(X_k)\|_8^2+3|{\rm Var}(F_n)-|H'(0)|^2|^2.
\end{align*}
We now note that $\|\tau_k(X_k)\|_8\geq \|\tau_k(X_k)\|_4\geq 1$ because $\E[\tau_k(X_k)]=1$ and $\|X_k\|_8\geq \|X_k\|_4\geq 1$ because $\E|X_k|^2=1.$ Hence, recalling (\ref{old1}), we conclude that there exists $C_H>0$ such that
$$\|\nabla_{F_n}-{\rm Var}(F_n)\|_4^2\leq \frac{C_H}{n}\max\limits_{1\leq k\leq n}\|X_k\|_8^4\max\limits_{1\leq k\leq n}\|\tau_k(X_k)\|_8^2.$$
This finishes the proof of (\ref{rold3}). The proof of the lemma is complete.
\end{proof}
\begin{lem}\label{re3}Let $\nabla_{F_n}$ be as in Lemma \ref{re2}. Assume that $\E|\tau_k(X_k)|^8+\E|\tau'_k(X_k)|^8<\infty$ for all $k\geq 1.$ Then, for all $n\geq 1,$  we have
\begin{equation}\label{y9io2}
\big\|\sum\limits_{k=1}^n \partial_k\nabla_{F_n}\tau_k(X_k)\big\|_4^2\leq C_H\max\limits_{1\leq k\leq n}(\|\tau'_k(X_k)\|_8^4+\|\tau_k(X_k)\|_8^4),
\end{equation}
where $C_H$ is a positive constant depending only on $\|H'\|_\infty$ and $\|H''\|_\infty.$
\end{lem}
\begin{proof} We recall that $\nabla_{F_n}=\frac{H'(0)H'(\bar{X})}{n}\sum\limits_{i=1}^n\tau_i(X_i),$ which gives us
$$\partial_k\nabla_{F_n}=\frac{H'(0)H''(\bar{X})}{n^2}\sum\limits_{i=1}^n\tau_i(X_i)+\frac{H'(0)H'(\bar{X})}{n}\tau_k'(X_k),\,\,1\leq k\leq n.$$
We therefore deduce
$$\sum\limits_{k=1}^n \partial_{k}\nabla_{F_n}\tau_k(X_k)=\frac{H'(0)H''(\bar{X})}{n^2}\big(\sum\limits_{k=1}^n\tau_k(X_k)\big)^2
+\frac{H'(0)H'(\bar{X})}{n}\sum\limits_{k=1}^n\tau_k'(X_k)\tau_k(X_k).$$
So we obtain  (\ref{y9io2}) because
\begin{align*}
\big\|\sum\limits_{k=1}^n \partial_k\nabla_F\tau_k(X_k)\big\|_4^2&\leq\frac{2\|H'\|_\infty^2\|H''\|_\infty^2}{n^4}\big\|\sum\limits_{k=1}^n\tau_k(X_k)\big\|_8^4
+\frac{2\|H'\|_\infty^4}{n^2}\big\|\sum\limits_{k=1}^n\tau_k'(X_k)\tau_k(X_k)\big\|_4^2\\
&\leq\frac{2\|H'\|_\infty^2\|H''\|_\infty^2}{n^4}\big(\sum\limits_{k=1}^n\|\tau_k(X_k)\|_8\big)^4
+\frac{2\|H'\|_\infty^4}{n^2}\big(\sum\limits_{k=1}^n\|\tau_k'(X_k)\tau_k(X_k)\|_4\big)^2\\
&\leq2\|H'\|_\infty^2\|H''\|_\infty^2\max\limits_{1\leq k\leq n}\|\tau_k(X_k)\|_8^4
+\|H'\|_\infty^4\max\limits_{1\leq k\leq n}(\|\tau'_k(X_k)\|_8^4+\|\tau_k(X_k)\|_8^4).
\end{align*}
The proof of the lemma is complete.
\end{proof}
We now are ready to state and prove the main result of this subsection.
\begin{thm}\label{kgl1}Consider the statistic $F_n$ defined by (\ref{i4}), where $H$ is a twice differentiable function with bounded derivatives and $H'(0)\neq 0.$ Assume that
$$\beta:=\max\big\{\sup\limits_{k\geq 1}\E|X_k|^8,\sup\limits_{k\geq 1}\E|\tau_k(X_k)|^8,\sup\limits_{k\geq 1}\E|\tau'_k(X_k)|^8\big\}<\infty.$$
Then, for $\sigma_n^2:={\rm Var}(F_n),$ we have
\begin{equation}\label{uir0}
|\sigma_n^2-|H'(0)|^2|\leq \frac{C_{H,\beta}}{\sqrt{n}}\,\,\forall\,n\geq 1
\end{equation}
and
\begin{equation}\label{uiomd}
I(\sigma_n^{-1}F_n\|Z)\leq \frac{C_{H,\beta}}{n\sigma_n^2}\left(1+\|\nabla_{F_n}^{-1}\|_8^2+\|\nabla_{F_n}^{-1}\|_8^4\right)\,\,\forall\,n\geq 1,
\end{equation}
where $C_{H,\beta}$ is a positive constant depending only on $\beta,\|H'\|_\infty$ and $\|H''\|_\infty$ and $\nabla_{F_n}$ is defined by
$$\nabla_{F_n}:=\frac{H'(0)H'(\bar{X})}{n}\sum\limits_{k=1}^n\tau_k(X_k).$$
\end{thm}
\begin{proof}The estimate (\ref{uir0}) follows directly from (\ref{old1}). On the other hand, we apply Theorem \ref{8uh4} to $F=\sigma_n^{-1}F_n,$ $p=q=4,r=s=t=2$ and $g_k(x)=\frac{H'(0)x}{\sigma_n\sqrt{n}}$ to get
\begin{align}
I(\sigma_n^{-1}F_n\|Z)&\leq \frac{3}{\sigma_n^2}\big\|F_n-\sqrt{n}H'(0)\bar{X}\big\|_2^2+\frac{3|H'(0)|^2n}{\sigma_n^2}\|\sigma_n^2-\nabla_{F_n}\|_4^2\|\bar{X}\|_8^2\|\nabla_{F_n}^{-1}\|_8^2\notag\\
&+\frac{3\sigma_n^2|H'(0)|^2}{n}\big\|\sum\limits_{k=1}^n \partial_{k}\nabla_{F_n}\tau_k(X_k)\big\|_4^2\|\nabla_{F_n}^{-1}\|_8^4.\notag
\end{align}
It follows from Lemmas \ref{re1}, \ref{re2} and \ref{re3} that there exists a positive constant $C_{H,\beta}$ such that
$$
I(\sigma_n^{-1}F_n\|Z)\leq \frac{C_{H,\beta}}{n\sigma_n^2}(1+\|\nabla_{F_n}^{-1}\|_8^2+\sigma_n^2\|\nabla_{F_n}^{-1}\|_8^4).
$$
So we get the desired bound (\ref{uiomd}) because $\sup\limits_{n\geq 1}\sigma_n^2<\infty.$ The proof of the theorem is complete.
\end{proof}
In view of Theorem \ref{kgl1}, we achieve a $O(1/n)$ convergence rate of the Fisher information distance if we additionally assume that
\begin{equation}\label{iom3}
\limsup\limits_{n\to \infty}\|\nabla_{F_n}^{-1}\|_8<\infty.
\end{equation}
In the next corollary we provide a simple condition ensuring (\ref{iom3}) to be fulfilled.
\begin{cor}\label{kmdl}In addition to the assumptions of Theorem \ref{kgl1}, we assume that $\inf\limits_{x\in \mathbb{R}}|H'(x)|>0$ and
\begin{equation}\label{hjy7}
\E|\tau_1(X_1) + \dots + \tau_{n_0}(X_{n_0})|^{-8}<\infty\,\,\,\text{for some}\,\,\,n_0\geq 1.
\end{equation}
Then, $\gamma:=\sup\limits_{n\geq n_0}n^8\E|\tau_1(X_1) + \dots + \tau_n(X_n)|^{-8}$ is a finite constant and we have  the following bound
\begin{equation}\label{uiomd6}
I(\sigma_n^{-1}F_n\|Z)\leq \frac{C_{H,\beta}\gamma^{\frac{1}{2}}}{n\sigma_n^2}\,\,\forall\,n\geq n_0,
\end{equation}
where $C_{H,\beta}$ is a positive constant depending only on $\beta,\|H'\|_\infty$ and $\|H''\|_\infty.$ %and $\inf\limits_{x\in \mathbb{R}}|H'(x)|.$
\end{cor}
\begin{proof}It follows from Proposition \ref{ujmld} that
$$\lim\limits_{n\to \infty}n^8\E|\tau_1(X_1) + \dots + \tau_n(X_n)|^{-8}=1.$$
Hence, $\gamma:=\sup\limits_{n\geq n_0}n^8\E|\tau_1(X_1) + \dots + \tau_n(X_n)|^{-8}$ is a finite constant. Now (\ref{uiomd6}) follows directly from the bound (\ref{uiomd}) and the fact that $|\nabla_{F_n}|\geq \frac{1}{n}\inf\limits_{x\in \mathbb{R}}|H'(x)|^2\sum\limits_{k=1}^n\tau_k(X_k).$
\end{proof}%In order to see clearer the novelty of our approach,
For the sake of completeness, we end this subsection by discussing the negative moment condition (\ref{hjy7}) required in Corollary \ref{kmdl}. Clearly, the condition (\ref{hjy7}) is self-satisfied if there exists $1\leq k\leq n_0$ such that $\tau_k(X_k)\geq c>0\,\,a.s.$ We also have the following simple criterion.
\begin{prop}\label{husu}Assume that the random variables $X_k,1\leq k\leq n_0$ satisfy $|\tau_k'(X_k)|\leq c_k\,\,a.s.$ where $c_k,1\leq k\leq n_0$ are positive real numbers and $\frac{1}{c_1^2}+\cdots+\frac{1}{c_{n_0}^2}>8.$ Then the condition (\ref{hjy7}) is satisfied.
\end{prop}
\begin{proof}Fixed $1\leq k\leq n_0,$ we consider the function
$$\varphi(x):=\E[e^{-x\tau_k(X_k)}],\,\,x\geq 0.$$
We have
$$\varphi'(x)=-\E[\tau_k(X_k)e^{-x\tau_k(X_k)}]=-\E[(\tau_k(X_k)-1)e^{-x\tau_k(X_k)}]-\varphi(x),\,\,x\geq 0.$$
Hence, by Lemma \ref{kflo}, we obtain
\begin{equation}\label{jkd1}
\varphi'(x)=x\E\left[e^{-x\tau_k(X_k)}\tau'_k(X_k)\frac{\int_{X_k}^\infty (\tau_k(y)-1)p_k(y)dy}{p_k(X_k)}\right]-\varphi(x),\,\,x\geq 0.
\end{equation}
By the assumption $|\tau_k'(X_k)|\leq c_k,$ we have
\begin{equation}\label{jkd2}
\bigg|\frac{\int_{X_k}^\infty (\tau_k(y)-1)p_k(y)dy}{p_k(X_k)}\bigg|\leq c_k \tau_k(X_k).
\end{equation}
Indeed, (\ref{jkd2}) follows from the following representations, see e.g. Corollary 2.1 in \cite{Saumard2018}
$$\tau_k(X_k)=\frac{\int_{-\infty}^\infty (\nu_k(X_k\wedge y)-\nu_k(X_k)\nu_k(y))dy}{p_k(X_k)},$$
$$\frac{\int_{X_k}^\infty (\tau_k(y)-1)p_k(y)dy}{p_k(X_k)}=\frac{\int_{-\infty}^\infty (\nu_k(X_k\wedge y)-\nu_k(X_k)\nu_k(y))\tau'_k(y)dy}{p_k(X_k)},$$
where $\nu_k$ denotes the cumulative distribution function of $X_k.$ We now combine (\ref{jkd1}) and (\ref{jkd2}) to get
\begin{align*}
\varphi'(x)&\leq xc_k^2\E\left[e^{-x\tau_k(X_k)}\tau_k(X_k)\right]-\varphi(x)\\
&= -xc_k^2\varphi'(x)-\varphi(x),\,\,\,x\geq 0.
\end{align*}
Solving the above differential inequality yields
\begin{equation}\label{iodl}
\E[e^{-x\tau_k(X_k)}]=\varphi(x)\leq (1+xc_k^2)^{-\frac{1}{c_k^2}},\,\,x\geq 0.
\end{equation}
So, recalling Remark \ref{ujiu}, we conclude that the condition (\ref{hjy7}) is satisfied. This finishes the proof.
\end{proof}
\subsection{Sum of non-i.i.d random variables}\label{89c}
In this subsection,  to see clearer the novelty of our approach, we discuss the Fisher information bounds for the sums of independent random variables. For simplicity, we consider the following sums of non-i.i.d random variables
$$S_n=\frac{X_1+\cdots+X_n}{\sqrt{n}},\,\,n\geq 1.$$
Note that the results of this subsection can be extended to the weighted sums without any new difficulty. We first apply Corollary \ref{kmdl} to $H(x)=x,$ and we obtain the following.
\begin{prop}\label{9ohj2}Assume that $\E|\tau_1(X_1) + \dots + \tau_{n_0}(X_{n_0})|^{-8}<\infty$ for some $n_0\geq 1$ and $$\beta:=\max\big\{\sup\limits_{k\geq 1}\E|X_k|^8,\sup\limits_{k\geq 1}\E|\tau_k(X_k)|^8,\sup\limits_{k\geq 1}\E|\tau'_k(X_k)|^8\big\}<\infty.$$
Then we have%$\gamma:=\sup\limits_{n\geq n_0}n^8\E|\tau_1(X_1) + \dots + \tau_n(X_n)|^{-8}<\infty$ and
\begin{equation}\label{0kd}
I(S_n\|Z)\leq \frac{C_{\beta,\bar{\beta}}(1+\E|\tau_1(X_1) + \dots + \tau_{n_0}(X_{n_0})|^{-8})^{\frac{1}{2}}}{n}\,\,\forall\,n\geq n_0,
\end{equation}
where $C_{\beta,\bar{\beta}}$ is a positive constant depending only on $\beta$ and $\bar{\beta}:=\sup\limits_{k\geq n_0+1}\E|\tau_k(X_k)|^2.$
\end{prop}
\begin{proof}In view of Remark \ref{u41u}, we have
$$\gamma=\sup\limits_{k\geq n_0}n^8\E|\tau_1(X_1) + \dots + \tau_n(X_n)|^{-8}\leq C_{\bar{\beta}} (1+\E|\tau_1(X_1) + \dots + \tau_{n_0}(X_{n_0})|^{-8}),$$
where $C_{\bar{\beta}}$ is a positive constant depending only on $\bar{\beta}:=\sup\limits_{k\geq n_0+1}\E|\tau_k(X_k)|^2.$ So (\ref{0kd}) follows directly from the bound (\ref{uiomd6}).
\end{proof}

\begin{rem}\label{jsm4}Since we are working in the setting of non-identical variables, the results of \cite{Artstein2004,Bobkov2014b,B-J,Johnson2020} can not be applied to $S_n.$  In Section 3.1 of \cite{Johnson2004}, Johnson obtained the following bound
\begin{equation}\label{0kds}
I(S_n\|Z)\leq \frac{1}{n}\sum\limits_{k=1}^n\frac{I(X_k\|Z)}{1+c_k},
\end{equation}
where $c_k:=\frac{1}{2R^\ast_k}\sum\limits_{i=1,i\neq k}^n\frac{1}{I(X_i)}$ and $R^\ast_k$ is the restricted Poincar\'e constant defined by
$$R^\ast_k:=\sup\limits_{g\in \mathcal{H}^\ast(X_k)}\frac{\E|g(X_k)|^2}{\E|g'(X_k)|^2},$$
where $\mathcal{H}^\ast(X_k)$ is the space of absolutely continuous functions $g$ such that $\E[g(X_k)]=\E[g'(X_k)]=0$ and $0<\E|g(X_k)|^2<\infty.$
\end{rem}
\begin{exam} Let us provide two examples to show that the bounds (\ref{0kd}) and (\ref{0kds}) are not the same (those bounds require different conditions to achieve the rate $O(1/n)$).\\
\noindent$(a)$ Denoted by $t_k$ Student's $t$-distribution with $\beta_k$ degrees of freedom ($\beta_k>16$), we consider the sequence $X_k:=\sqrt{\frac{\beta_k-2}{\beta_k}}\,t_k,k\geq 1.$ We have
$$\text{$\E[X_k]=0,{\rm Var}(X_k)=1$ and $\tau_k(X_k)=\frac{X_k^2+\beta_k-2}{\beta_k-1}$.}$$
It is easy to verify that all conditions of Proposition \ref{9ohj2} are fulfilled. In particular, we can take $n_0=1$ because $\tau_1(X_1)=\frac{X_1^2+\beta_1-2}{\beta_1-1}\geq \frac{\beta_1-2}{\beta_1-1}>\frac{14}{15}.$ So the bound  (\ref{0kd}) gives us the following
$$I(S_n\|Z)\leq \frac{C_{\beta,\bar{\beta}}}{n}\,\,\forall\,n\geq 1.$$
Meanwhile, the bound  (\ref{0kds}) fails to apply to the sequence $(X_k)_{k\geq 1}$ (note that the finiteness of Poincar\'e constant $R^\ast_k$ implies that all moments of $X_k$ are finite). Thus our bound (\ref{0kd}) supplements the bound (\ref{0kds}).\\
\noindent$(b)$ Denoted by $\chi_k^2$ the chi-squared distribution with $d_k$ degrees of freedom ($d_k\geq 2$), we consider the sequence $X_k:=\frac{\chi_k^2-d_k}{\sqrt{2d_k}},k\geq 1.$ Once again, the bound  (\ref{0kds}) fails to apply to the sequence $(X_k)_{k\geq 1}$ because the chi-squared distribution do not satisfy  the Poincar\'e inequality. We have
$$\text{$\E[X_k]=0,{\rm Var}(X_k)=1$ and $\tau_k(X_k)=\frac{\sqrt{2d_k}X_k+d_k}{d_k}$.}$$
We observe that $|\tau'_k(X_k)|=\sqrt{2/d_k}=:c_k,k\geq 1$ and $\frac{1}{c_1^2}+\cdots+\frac{1}{c_{9}^2}>8.$ Thanks to Proposition \ref{husu} we get
$$\E|\tau_1(X_1) + \dots + \tau_{9}(X_{9})|^{-8}<\infty.$$
This allows us to use the bound (\ref{0kd}) with $n_0=9$ and we obtain
$$I(S_n\|Z)\leq \frac{C_{\beta,\bar{\beta}}(1+\E|\tau_1(X_1) + \dots + \tau_{9}(X_{9})|^{-8})^{\frac{1}{2}}}{n}\,\,\forall\,n\geq 9.$$
Furthermore, by the negative moment formula (\ref{iihg}) and the estimate (\ref{iodl}), we have
\begin{align*}
\E|\tau_1(X_1) + \dots + \tau_{9}(X_{9})|^{-8}&=\frac{1}{\Gamma(8)}\int_0^\infty x^{7}\E[e^{-x(\tau_1(X_1) + \dots + \tau_{9}(X_{9}))}]dx\\
&\leq \frac{1}{\Gamma(8)}\int_0^\infty x^{7}\prod\limits_{k=1}^{9}(1+2x/d_k)^{-\frac{d_k}{2}}dx.
\end{align*}
\end{exam}
%\begin{rem}
To achieve the rate $O(1/n)$ for the sums $Z_n$ of i.i.d random variables, the best possible hypothesis is that $I(Z_{n_0})<\infty$ for some $n_0\geq 1$ and $\E|X_1|^4<\infty,$ see \cite{Bobkov2014b}. In the next theorem, by employing the special structure of the sums of independent random variables, we are able to obtain an analogous hypothesis for the sums of non-i.i.d random variables.

%the authors of \cite{Bobkov2014b} only require the hypothesis that $I(Z_{n_0})<\infty$ for some $n_0\geq 1$ and $\E|X_1|^4<\infty.$ This hypothesis is the best possible one.
%$$I(Z_n\|Z)\leq \frac{c}{n}$$
%under the hypothesis that $I(Z_{n_0})<\infty$ for some $n_0\geq 1$ and $\E|X_1|^4<\infty.$
%\end{rem}
\begin{thm}\label{io9y}
Assume that there exists $n_0\geq 1$ such that $I(S_{n_0})<\infty$ and
$$\beta:=\max\big\{\sup\limits_{k\geq n_0+1}\E|X_k|^4,\sup\limits_{k\geq n_0+1}\E|\tau_k(X_k)|^8,\sup\limits_{k\geq n_0+1}\E|\tau'_k(X_k)|^8\big\}<\infty.$$
Then we have
$$I(S_n\|Z)\leq \frac{C_{n_0,\beta,\bar{\beta}}}{n}\,\,\forall\,n\geq n_0,$$
where $C_{n_0,\beta,\bar{\beta}}$ is a positive constant depending only on $n_0,\beta$ and  $\bar{\beta}:=\sup\limits_{k\geq n_0+1}\E|\tau_k(X_k)|^2.$
\end{thm}
\begin{proof} We write $S_n=\frac{\sqrt{n_0}S_{n_0}+X_{n_0+1}+\cdots+X_n}{\sqrt{n}},\,\,n\geq n_0$ and consider $S_n$ as a function of $(S_{n_0},X_{n_0+1},\cdots,X_n).$ Define the functions $g_{n_0}(S_{n_0})=-\sqrt{n_0/n}\rho_{S_{n_0}}(S_{n_0})$ and $g_k(X_k)=\frac{X_k}{\sqrt{n}}$ for $n_0+1\leq k\leq n.$ In view of Proposition \ref{klf22} we have
$$\nabla_{S_n}=\frac{n_0+\tau_{n_0+1}(X_{n_0+1})+\cdots+\tau_n(X_n)}{n}$$
and
$$\rho_{S_n}(x)=-\E\left[\frac{-\sqrt{n_0}\rho_{S_{n_0}}(S_{n_0})+\sum\limits_{k=n_0+1}^n X_k}{\sqrt{n}\nabla_{S_n}}+\frac{\sum\limits_{k=n_0+1}^n \tau'_k(X_k)\tau_k(X_k)}{n\sqrt{n}\nabla_{S_n}^2}\bigg|F=x\right].$$
Using the same arguments as in the proof of (\ref{olfm3}) we get
\begin{align*}
I(S_n\|Z) &\leq \frac{3n_0}{n}I(S_{n_0}\|Z)\notag\\
&+ \frac{3}{n}\E\bigg|\frac{\big(\sqrt{n_0}\rho_{S_{n_0}}(S_{n_0})-\sum\limits_{k=n_0+1}^n X_k\big)(1-\nabla_{S_n})}{\nabla_{S_n}}\bigg|^2+\frac{3}{n^3}\E\bigg|\frac{\sum\limits_{k=n_0+1}^n \tau'_k(X_k)\tau_k(X_k)}{\nabla_{S_n}^2}\bigg|^2,
\end{align*}
which implies that
\begin{align*}
I(S_n\|Z)
&\leq \frac{3n_0}{n}I(S_{n_0}\|Z)+\frac{3n_0}{n}I(S_{n_0})\E\bigg|\frac{1-\nabla_{S_n}}{\nabla_{S_n}}\bigg|^2\notag\\
&+ \frac{3}{n}\E\bigg|\frac{\sum\limits_{k=n_0+1}^n X_k(1-\nabla_{S_n})}{\nabla_{S_n}}\bigg|^2+\frac{3}{n^3}\E\bigg|\frac{\sum\limits_{k=n_0+1}^n \tau'_k(X_k)\tau_k(X_k)}{\nabla_{S_n}^2}\bigg|^2.
\end{align*}
Hence, by H\"older's inequality, we deduce
\begin{align*}
I(S_n\|Z)&\leq \frac{3n_0}{n}I(S_{n_0}\|Z)+\frac{3n_0}{n}I(S_{n_0})\|1-\nabla_{S_n}\|_4^2\|\nabla_{S_n}^{-1}\|_8^2\notag\\
&+ \frac{3}{n}\big\|\sum\limits_{k=n_0+1}^n X_k\big\|_4^2\|1-\nabla_{S_n}\|_8^2\|\nabla_{S_n}^{-1}\|_8^2+\frac{3}{n^3}\big\|\sum\limits_{k=n_0+1}^n \tau'_k(X_k)\tau_k(X_k)\big\|_4^2\|\nabla_{S_n}^{-1}\|_8^4.
\end{align*}
By using the inequality (\ref{9hj4}) we have $\big\|\sum\limits_{k=n_0+1}^n X_k\big\|_4^2\leq C_\beta n$ and $\|1-\nabla_{S_n}\|_4^2\leq \|1-\nabla_{S_n}\|_8^2\leq \frac{C_\beta}{n}$ for some $C_\beta>0.$ We also have $\big\|\sum\limits_{k=n_0+1}^n \tau'_k(X_k)\tau_k(X_k)\big\|_4^2\leq C_\beta n^2$ which was proved in the proof of Lemma \ref{re3}. On the other hand, in view of Remark \ref{u41u}, we have
$$\sup\limits_{n\geq n_0}\|\nabla_{S_n}^{-1}\|_8^8=\sup\limits_{n\geq n_0}n^8\E|n_0+\tau_{n_0+1}(X_{n_0+1}) + \dots + \tau_n(X_n)|^{-8}\leq C_{\bar{\beta}}n_0^{-8},$$
where $C_{\bar{\beta}}$ is a positive constant depending only on $\bar{\beta}:=\sup\limits_{k\geq n_0+1}\E|\tau_k(X_k)|^2.$%which points out that $\gamma:=\sup\limits_{n\geq n_0}n^8\E|n_0+\tau_{n_0+1}(X_{n_0+1}) + \dots + \tau_n(X_n)|^{-8}<\infty.$

So we conclude that there exists a positive constant $C_{n_0,\beta,\bar{\beta}}$ depending only on $n_0,\beta$ and $\bar{\beta}$ such that
$$I(S_n\|Z)\leq \frac{C_{n_0,\beta,\bar{\beta}}}{n}\,\,\forall\,n\geq n_0.$$
The proof of the theorem is complete.
\end{proof}

\section{Further remarks}\label{jkfm}%and
The problem of obtaining quantitative estimates for the Fisher information distance is a challenge to solve. In this paper, we successfully constructed a method to investigate this problem in a general framework of nonlinear statistics. Our Theorems \ref{gko}, \ref{8uh4}, \ref{uij8} and \ref{kgl1} point out that we only need to estimate negative and positive moments of certain statistics to obtain the Fisher information bounds. Our method thus makes a connection between the Fisher information and theory of moment inequalities, which is very rich. %This is the main advantage of our method.

In this Section, we first recall the relationship between forms of convergence to show the usefulness of the Fisher information bounds. We then show the universality of our approach by providing the Fisher information bounds in non-central limit theorems.
\subsection{Fisher information and other forms of convergence}\label{ufk2}
Since the Fisher information distance is a very strong measure of non-Gaussianity, our Fisher information bounds can be used to derive quantitative bounds on several weaker distances. Let $F$ be a random variable with $\E[F]=0,{\rm Var}(F)=1$ and a density $p_F.$ Denote by $\phi$ the density of standard normal random variable $Z.$ We have the following relationship between convergence in Fisher information and some weaker forms of convergence.
\begin{itemize}
\item Uniform distance between densities, see \cite{Shimizu1975}:
$$U(F\|Z):=\sup\limits_{x\in \mathbb{R}}|p_F(x)-\phi(x)|\leq (1+\sqrt{6/\pi})\sqrt{I(F\|Z)}.$$
\item The relative entropy or Kullback-Leibler distance, see \cite{Stam1959}:
$$D(F\|Z):=\E\left[\log \frac{p_F(F)}{\phi(F)}\right]\leq \frac{1}{2}I(F\|Z).$$
\item The Kantorovich distance or the quadratic Wasserstein distance, see \cite{Talagrand1996}:
$$W^2(F\|Z):=\inf\{\E|F'-Z'|^2: F'\sim F, Z'\sim Z\}\leq 2 D(F\|Z),$$
where the infimum is taken over all random variables $F'$ and $Z'$ having the same distribution as that of $F$ and $Z,$ respectively.
\item Total variation distance, see \cite{Pinsker1964}: %by the Pinsker-Csisz\'ar-Kullback inequality
$$d_{TV}(F\|Z):=\sup\limits_{A}|P(F\in A)-P(Z\in A)|\leq \sqrt{2D(F\|Z)},$$
where the supremum is running over all Borel subsets $A$ of the real line.
\end{itemize}
Note that, in the framework of nonlinear statistics, the total variation bounds can be found in \cite{Nicolas2018}. The distances $W^2(F\|Z)$ and $D(F\|Z)$ can be investigated by using the results obtained in \cite{Ledoux2015}. However, to the best of our knowledge, it remains an open problem to estimate the distances $U(F\|Z)$ without using the Fisher information bounds obtained in the present paper. In the next theorem, we provide a new application of Fisher information bounds.
\begin{thm}[Nonuniform Berry-Esseen bounds] For any $x\in \mathbb{R},$ we have
\begin{equation}\label{ggh}
|P(F\geq x)-P(Z\geq x)|\leq \sqrt{I(F\|Z)/2}\left(\sqrt{P(|F|\geq |x|)}+\sqrt{P(|Z|\geq |x|)}\right).
\end{equation}
\end{thm}%The inequality (\ref{ggh}) can be proved as follows.
\begin{proof}We put $h(x)=p_F(x)/\phi(x).$ Then, we have
\begin{align*}
|P(F\geq x)-P(Z\geq x)|
&=|\E[(h(Z)-1)\ind_{\{Z\geq x\}}]|\\
&\leq \E[|\sqrt{h(Z)}-1|\sqrt{h(Z)}\ind_{\{Z\geq x\}}]+\E[|\sqrt{h(Z)}-1|\ind_{\{Z\geq x\}}].
\end{align*}
By using the Cauchy-Schwarz inequality we deduce
\begin{align}
|P(F\geq x)-P(Z\geq x)|&\leq \sqrt{\E|\sqrt{h(Z)}-1|^2}\left(\sqrt{\E[h(Z)\ind_{\{Z\geq x\}}]}+\sqrt{\E[\ind_{\{Z\geq x\}}]}\right)\notag\\
&= \sqrt{\E|\sqrt{h(Z)}-1|^2}\left(\sqrt{P(F\geq x)}+\sqrt{P(Z\geq x)}\right).\label{8dk}
\end{align}
We now observe that $\E h(Z)=1$ and $\E \sqrt{h(Z)}\leq \sqrt{\E h(Z)}=1.$ Hence, we obtain
\begin{align*}
\E|\sqrt{h(Z)}-1|^2&=2\E h(Z)-2\E\sqrt{h(Z)}\\
&\leq 2\E h(Z)-2(\E\sqrt{h(Z)})^2\\
&=2{\rm Var}(\sqrt{h(Z)}).
\end{align*}
Furthermore, by Gaussian Poincar\'e inequality, we have
\begin{align}
\E|\sqrt{h(Z)}-1|^2&\leq \frac{1}{2}\E\left[\frac{|h'(Z)|^2}{h(Z)}\right]\notag\\
&= \frac{1}{2}\E\left[\frac{\phi(Z)}{p_F(Z)}\bigg(\frac{p_F'(Z)}{\phi(Z)}+\frac{Zp_F(Z)}{\phi(Z)}\bigg)^2\right]\notag\\
&= \frac{1}{2}\E\left[\frac{p_F(Z)}{\phi(Z)}\bigg(\frac{p_F'(Z)}{p_F(Z)}+Z\bigg)^2\right]\notag\\
&=\frac{1}{2}I(F\|Z)\label{8dka}
\end{align}
Inserting (\ref{8dka}) into (\ref{8dk}) yields
\begin{equation}\label{uja}
|P(F\geq x)-P(Z\geq x)|\leq \sqrt{I(F\|Z)/2}\left(\sqrt{P(F\geq x)}+\sqrt{P(Z\geq x)}\right).
\end{equation}
Similarly, since $|P(F\geq x)-P(Z\geq x)|=|\E[(h(Z)-1)\ind_{\{Z\leq x\}}],$ we also have
\begin{equation}\label{ujb}
|P(F\geq x)-P(Z\geq x)|\leq \sqrt{I(F\|Z)/2}\left(\sqrt{P(F\leq x)}+\sqrt{P(Z\leq x)}\right).
\end{equation}
So we obtain (\ref{ggh}) by combining (\ref{uja}) and (\ref{ujb}).
\end{proof}
The significance of the inequality (\ref{ggh}) lies in the fact that it allows us to employ theory of concentration inequalities for deriving nonuniform Berry-Esseen bounds. We consider, for example, the normalized sums $Z_n$ defined by (\ref{iklf}). We assume $\E |X_1|^4 < +\infty$
and $I(Z_{n_0}) < +\infty$ for some $n_0.$ Then, by the results of \cite{Bobkov2014b}, $I(Z_n\|Z)\leq C/n$ for some $C>0.$ We now additionally assume that $X_1$ is a sub-Gaussian random variable. Then, for some $c>0,$ we have
$$P(|Z_n|\geq |x|)\leq 2e^{-x^2/2c},\,\,x\in \mathbb{R}.$$
Note that $P(|Z|\geq |x|)\leq e^{-x^2/2},x\in \mathbb{R}.$ Hence, for some $C,c>0,$ we get
\begin{equation}\label{kglf}
|P(Z_n\geq x)-P(Z\geq x)|\leq \frac{Ce^{-x^2/2c}}{\sqrt{n}},\,\,x\in \mathbb{R}.
\end{equation}
It is worth to note that the bound (\ref{kglf}) is much shaper than the classical one that reads
$$|P(Z_n\geq x)-P(Z\geq x)|\leq \frac{C\E|X_1|^3}{\sqrt{n}(1+|x|^3)},\,\,x\in \mathbb{R}.$$
\subsection{Fisher information and non-central limit theorems}\label{ufk2b}
It is well known that the nonlinear statistics $F$ defined by (\ref{klf7}) do not always satisfy the central limit theorem, i.e. its asymptotic distribution can be non-normal.  Let $Y$ denote the asymptotic distribution of $F.$ The different distances between the distributions of $F$ and $Y$ have been extensively investigated. Among others we cite \cite{Chatterjee2011} and references therein for Berry-Esseen error bounds, \cite{Johnson2022} for the relative entropy results. Here we deal with the Fisher information distance of $F$ to $Y$ defined by
$$I(F\|Y)=\E[(\rho_F(F)-\rho_Y(F))^2].$$
We assume that the asymptotic distribution $Y$ has a differentiable density $p_Y$ supported on some connected interval of $\mathbb{R}$ and, for simplicity, $\E[Y]=0.$ Under such assumptions, the Stein kernel of $Y$ is given by
$$\tau_Y(x)=\frac{\int_{x}^\infty yp_Y(y)dy}{p_Y(x)}$$
for $x$ in the support of $p_Y.$ Furthermore, we have
$$\tau'_Y(x)=-x-\frac{p'_Y(x)\tau_Y(x)}{p_Y(x)},$$
and hence, the score function $\rho_Y$ of $Y$ can be represented by
$$\rho_Y(x)=-\frac{x+\tau'_Y(x)}{\tau_Y(x)}.$$
Now our approach can be used to bound the Fisher information distance as follows.
\begin{thm}\label{ema}Let $F\in \mathcal{C}_{\mathbb{X}_n}^2$ be such that $\E[F]=0.$ Then, for any differentiable decomposition $\{\mathcal{M}_1F,\cdots,\mathcal{M}_nF\}$ of $F$ with $\Theta_F\neq 0$ a.s. and for $p,q,r,s,t>1$ with $\frac{1}{p}+\frac{1}{q}+\frac{1}{r}=\frac{1}{s}+\frac{1}{t}=1,$ we have
\begin{multline}\label{fujks}
I(F\|Y)\leq 3(\|\rho_Y(F)\|_{2p}^2\|\Theta_F^{-1}\|_{2q}^2+\|\tau'_Y(F)\|_{2p}^2\|\Theta_F^{-1}\|_{4q}^4)\|\tau_Y(F)-\Theta_F\|_{2r}^2\\
+3\|\Theta_F^{-1}\|_{4s}^4\|\Theta_{\Theta_F,F}-\tau'_Y(F)\tau_Y(F)\|_{2t}^2,
\end{multline}
provided that the norms in the right hand side are finite.
\end{thm}
\begin{proof}
By the representation formula (\ref{d3t13}) we deduce
\begin{align}
I(F\|Y)&=\E[(\rho_F(F)-\rho_Y(F))^2]\notag\\
&\leq\E\left|\frac{F}{\Theta_F}+\frac{\Theta_{\Theta_F,F}}{\Theta_F^2}-\frac{F+\tau'_Y(F)}{\tau_Y(F)}\right|^2\notag\\
&=\E\left|\frac{(F+\tau'_Y(F))(\tau_Y(F)-\Theta_F)}{\Theta_F\tau_Y(F)}+\frac{\Theta_{\Theta_F,F}-\tau'_Y(F)\tau_Y(F)}{\Theta_F^2}
+\frac{\tau'_Y(F)(\tau_Y(F)-\Theta_F)}{\Theta_F^2}\right|^2\notag\\
&=\E\left|\frac{\rho_Y(F)(\tau_Y(F)-\Theta_F)}{\Theta_F}+\frac{\Theta_{\Theta_F,F}-\tau'_Y(F)\tau_Y(F)}{\Theta_F^2}
+\frac{\tau'_Y(F)(\tau_Y(F)-\Theta_F)}{\Theta_F^2}\right|^2.\notag
\end{align}
Then, by Cauchy-schwarz inequality, we obtain
\begin{align}
I(F\|Y)
&\leq 3\E\left|\frac{\rho_Y(F)(\tau_Y(F)-\Theta_F)}{\Theta_F}\right|^2+3\E\left|\frac{\Theta_{\Theta_F,F}-\tau'_Y(F)\tau_Y(F)}{\Theta_F^2}\right|^2\notag\\
&+3\E\left|\frac{\tau'_Y(F)(\tau_Y(F)-\Theta_F)}{\Theta_F^2}\right|^2.\notag
\end{align}
For $p,q,r,s,t>1$ with $\frac{1}{p}+\frac{1}{q}+\frac{1}{r}=\frac{1}{s}+\frac{1}{t}=1,$ we use H\"older's inequality to get
\begin{multline*}
I(F\|Y)\leq 3(\|\rho_Y(F)\|_{2p}^2\|\Theta_F^{-1}\|_{2q}^2+\|\tau'_Y(F)\|_{2p}^2\|\Theta_F^{-1}\|_{4q}^4)\|\tau_Y(F)-\Theta_F\|_{2r}^2\\
+3\|\Theta_F^{-1}\|_{4s}^4\|\Theta_{\Theta_F,F}-\tau'_Y(F)\tau_Y(F)\|_{2t}^2.
\end{multline*}
This completes the proof of the theorem.
\end{proof}
The detailed study of Fisher information bounds in non-central limit theorems is beyond the scope of this paper. Let us only provide an example to demonstrate the validity of the bound (\ref{fujks}).
%\begin{exam}Let $U_1,U_2,\cdots,$ be i.i.d. random variables uniformly distributed on the interval $(0,1)$ and $X_{1n}:=\min\{U_1,\cdots,U_n\},n\geq 1.$
%\end{exam}
\begin{exam}Let $X_1,X_2,\cdots,X_d$ be independent standard normal random variables. We consider the sequence
$$F_n:=\lambda_{n1}(X_1^2-1)+\cdots+\lambda_{nd}(X_d^2-1),\,\,n\geq 1,$$
where $\lambda_{nk}$ are real numbers satisfying $\lim\limits_{n\to\infty}\lambda_{nk}=1,1\leq k\leq d.$ It is easy to see that $F_n$ converges in distribution to $Y:=X_1^2+\cdots+X_d^2-d$ as $n\to\infty.$ The random variable $Y+d$ has the chi-squared distribution with $d$ degrees of freedom and hence, the Stein kernel of $Y$ can be computed by
$$\tau_Y(x)=2x+2d,\,\,x\geq -d.$$
To investigate the Fisher information distance, we additionally impose the following conditions

\noindent$(A_1)$ $d\geq 9,$

\noindent$(A_2)$ $\lambda_{nk}\geq 0$ and $\lambda_{n1}+\cdots+\lambda_{nd}\leq d$ for all $1\leq k\leq d$ and $n\geq 1.$

\noindent Note that the condition $(A_1)$ is used to control negative moments. Meanwhile, $(A_2)$ is a necessary condition ensuring $F_n$ belongs to the range of $Y.$ We consider $F_n$ as a function of $(X_1,X_2,\cdots,X_d)$ and use the following decomposition
$$\mathcal{M}_1F_n=\lambda_{n1}(X_1^2-1),\cdots,\mathcal{M}_dF_n=\lambda_{nd}(X_d^2-1).$$
By straightforward computations we obtain
$$\Theta_{F_n}=2\lambda_{n1}^2X_1^2+\cdots+2\lambda_{nd}^2X_d^2$$
and
$$\Theta_{\Theta_{F_n},F_n}=4\lambda_{n1}^3X_1^2+\cdots+4\lambda_{nd}^3X_d^2.$$
We have
\begin{align*}
\Theta_{F_n}-\tau_Y(F_n)&=2(\lambda_{n1}^2-\lambda_{n1})X_1^2+\cdots+2(\lambda_{nd}^2-\lambda_{nd})X_d^2
+2(\lambda_{n1}+\cdots+\lambda_{nd}-d)
\end{align*}
and we get
\begin{align}
\|\Theta_{F_n}-\tau_Y(F_n)\|_{2r}&\leq 2\|(\lambda_{n1}^2-\lambda_{n1})X_1^2\|_{2r}+\cdots+2\|(\lambda_{nd}^2-\lambda_{nd})X_d^2\|_{2r}
+2|\lambda_{n1}+\cdots+\lambda_{nd}-d|\notag\\
&\leq 2\big(\max\limits_{1\leq k\leq d}\lambda_{nk}\|X_1^2\|_{2r}+1\big)\sum\limits_{k=1}^d|\lambda_{nk}-1|\,\,\forall\,r>1.\label{amkd}
\end{align}
Similarly, we have
$$
\Theta_{\Theta_{F_n},F_n}-\tau'_Y(F_n)\tau_Y(F_n)=4(\lambda_{n1}^3-\lambda_{n1})X_1^2+\cdots+4(\lambda_{nd}^3-\lambda_{nd})X_d^2
+4(\lambda_{n1}+\cdots+\lambda_{nd}-d)
$$
and
\begin{equation}
\|\Theta_{\Theta_F,F}-\tau'_Y(F)\tau_Y(F)\|_{2t}\leq 4\big(\max\limits_{1\leq k\leq d}(\lambda_{nk}^2+\lambda_{nk})\|X_1^2\|_{2t}+1\big)\sum\limits_{k=1}^d|\lambda_{nk}-1|\,\,\forall\,t>1.
\end{equation}
Recalling the negative moment formula (\ref{iihg}) and the fact $\E[e^{-xX_k^2}]= (1+2 x)^{-\frac{1}{2}}$ for all $x\geq 0$ and $1\leq k\leq d,$ we have
$$c_{\alpha,d}:=\|(X_1^2 + \dots + X_d^2)^{-1}\|_\alpha<\infty\,\,\forall\,\alpha<d/2.$$
We therefore obtain the following estimates
\begin{equation}\label{uids}
\|\Theta_{F_n}^{-1}\|_{2q}\leq \frac{c_{2q,d}}{2\min\limits_{1\leq k\leq d}\lambda_{nk}^2}\,\,\forall\,1<q<d/4,
\end{equation}
\begin{equation}\label{uids1}
\|\Theta_{F_n}^{-1}\|_{4s}\leq \frac{c_{4s,d}}{2\min\limits_{1\leq k\leq d}\lambda_{nk}^2}\,\,\forall\,1<s<d/8.
\end{equation}
Let $p_1,p_2>p\in(1,d/4)$ be such that $\frac{1}{p_1}+\frac{1}{p_2}=\frac{1}{p}$ we have
\begin{align}
\|\rho_Y(F_n)\|_{2p}&=\big\|\frac{F_n+2}{2F_n+2d}\big\|_{2p}\leq \frac{1}{2}\|(F_n+d)^{-1}\|_{2p_1}\|F_n+2\|_{2p_2}\notag\\
&\leq\frac{c_{2p_1,d}\big(d\max\limits_{1\leq k\leq d}\lambda_{nk}\|X_1^2-1\|_{2p_2}+2\big)}{2\min\limits_{1\leq k\leq d}\lambda_{nk}}\,\,\forall\,p<p_1<d/4.\label{amkd1}
\end{align}
Let $\varepsilon>0$ be sufficiently small (for example $\varepsilon=1/16$). We choose to use $p=q=2+\varepsilon,$ $\frac{1}{r}=1-\frac{2}{p},$ $s=1+\varepsilon,$ $\frac{1}{t}=1-\frac{1}{s}$ and $p_1=2+2\varepsilon,$ $\frac{1}{p_2}=\frac{1}{p}-\frac{1}{p_1}.$ Then, Theorem \ref{ema} gives us
\begin{multline*}%$\frac{1}{r}\frac{1}{p}+\frac{1}{q}+=\frac{1}{s}+\frac{1}{t}=1,$
I(F_n\|Y)\leq 3(\|\rho_Y(F_n)\|_{2p}^2\|\Theta_{F_n}^{-1}\|_{2q}^2+\|\tau'_Y(F_n)\|_{2p}^2\|\Theta_{F_n}^{-1}\|_{4q}^4)\|\tau_Y(F_n)-\Theta_{F_n}\|_{2r}^2\\
+3\|\Theta_{F_n}^{-1}\|_{4s}^4\|\Theta_{\Theta_{F_n},F_n}-\tau'_Y(F_n)\tau_Y(F_n)\|_{2t}^2.
\end{multline*}
So, by combining the estimates (\ref{amkd})-(\ref{amkd1}), we conclude that there exists a positive constant $C_{\varepsilon,d}$ such that
$$I(F_n\|Y)\leq \frac{C_{\varepsilon,d}\big(\max\limits_{1\leq k\leq d}(\lambda_{nk}^2+\lambda_{nk})+1\big)^2}{\min\limits_{1\leq k\leq d}\lambda_{nk}^8}\big(\sum\limits_{k=1}^d|\lambda_{nk}-1|\big)^2\,\,\forall\,n\geq 1.$$
In particular, under the conditions $(A_1)$ and $(A_2),$ we have $I(F_n\|Y)\to 0$ as $n\to\infty$ if $\lim\limits_{n\to\infty}\lambda_{nk}=1,1\leq k\leq d.$
\end{exam}
%\noindent {\bf Acknowledgments.} I am grateful to Professor O. Johnson for insightful discussions.
\noindent {\bf Acknowledgments.}  The authors would like to thank the anonymous referees for valuable comments which led to the improvement of this work.

\noindent {\bf Declarations.} Conflict of interest: The authors declare no competing interests.
%\noindent {\bf Conflict of interest statement.} The authors have no competing interests to declare that are relevant to the content of this article.

\noindent {\bf Data Availability Statement.} Data sharing not applicable to this article as no datasets were generated or analysed during the current study.

\end{document}